\documentclass[11pt]{amsart}

\usepackage{etex}
\usepackage{amsmath}
\usepackage{amssymb}
\usepackage{amscd}
\usepackage{amsthm}
\usepackage{hyperref}
\usepackage{mathtools}
\usepackage{extarrows}
\usepackage[all]{xy}
\xyoption{color}

\usepackage{tikz}
\usepackage{pgfplots}
\pgfplotsset{compat=1.9}
\usepackage{tikz-cd}

\newtheorem{thm}{Theorem}[section]
\newtheorem{lem}[thm]{Lemma}

\newtheorem{conj}[thm]{Conjecture}
\newtheorem{prop}[thm]{Proposition}

\theoremstyle{remark}

\newtheorem{rem}[thm]{Remark}

\theoremstyle{definition}
\newtheorem{defn}[thm]{Definition}

\numberwithin{equation}{section}

\allowdisplaybreaks[4]

\begin{document}

\vfuzz0.5pc
\hfuzz0.5pc 

\newcommand{\claimref}[1]{Claim \ref{#1}}
\newcommand{\thmref}[1]{Theorem \ref{#1}}
\newcommand{\propref}[1]{Proposition \ref{#1}}
\newcommand{\lemref}[1]{Lemma \ref{#1}}
\newcommand{\coref}[1]{Corollary \ref{#1}}
\newcommand{\remref}[1]{Remark \ref{#1}}
\newcommand{\conjref}[1]{Conjecture \ref{#1}}
\newcommand{\questionref}[1]{Question \ref{#1}}
\newcommand{\defnref}[1]{Definition \ref{#1}}
\newcommand{\secref}[1]{\S \ref{#1}}
\newcommand{\ssecref}[1]{\ref{#1}}
\newcommand{\sssecref}[1]{\ref{#1}}

\newcommand{\RED}{{\mathrm{red}}}
\newcommand{\tors}{{\mathrm{tors}}}
\newcommand{\eq}{\Leftrightarrow}

\newcommand{\mapright}[1]{\smash{\mathop{\longrightarrow}\limits^{#1}}}
\newcommand{\mapleft}[1]{\smash{\mathop{\longleftarrow}\limits^{#1}}}
\newcommand{\mapdown}[1]{\Big\downarrow\rlap{$\vcenter{\hbox{$\scriptstyle#1$}}$}}
\newcommand{\smapdown}[1]{\downarrow\rlap{$\vcenter{\hbox{$\scriptstyle#1$}}$}}

\newcommand{\A}{{\mathbb A}}
\newcommand{\I}{{\mathcal I}}
\newcommand{\J}{{\mathcal J}}
\newcommand{\CO}{{\mathcal O}}
\newcommand{\C}{{\mathcal C}}
\newcommand{\BC}{{\mathbb C}}
\newcommand{\BQ}{{\mathbb Q}}
\newcommand{\m}{{\mathcal M}}
\newcommand{\h}{{\mathcal H}}
\newcommand{\Z}{{\mathcal Z}}
\newcommand{\BZ}{{\mathbb Z}}
\newcommand{\W}{{\mathcal W}}
\newcommand{\Y}{{\mathcal Y}}
\newcommand{\T}{{\mathcal T}}
\newcommand{\BP}{{\mathbb P}}
\newcommand{\CP}{{\mathcal P}}
\newcommand{\G}{{\mathbb G}}
\newcommand{\BR}{{\mathbb R}}
\newcommand{\D}{{\mathcal D}}
\newcommand{\DD}{{\mathcal D}}
\newcommand{\LL}{{\mathcal L}}
\newcommand{\f}{{\mathcal F}}
\newcommand{\E}{{\mathcal E}}
\newcommand{\BN}{{\mathbb N}}
\newcommand{\N}{{\mathcal N}}
\newcommand{\K}{{\mathcal K}}
\newcommand{\R} {{\mathbb R}}
\newcommand{\PP}{{\mathbb P}}
\newcommand{\Pp}{{\mathbb P}}
\newcommand{\BF}{{\mathbb F}}
\newcommand{\QQ}{{\mathcal Q}}
\newcommand{\closure}[1]{\overline{#1}}
\newcommand{\EQ}{\Leftrightarrow}
\newcommand{\imply}{\Rightarrow}
\newcommand{\isom}{\cong}
\newcommand{\embed}{\hookrightarrow}
\newcommand{\tensor}{\mathop{\otimes}}
\newcommand{\wt}[1]{{\widetilde{#1}}}
\newcommand{\ol}{\overline}
\newcommand{\ul}{\underline}

\newcommand{\bs}{{\backslash}}
\newcommand{\CS}{{\mathcal S}}
\newcommand{\CA}{{\mathcal A}}
\newcommand{\Q} {{\mathbb Q}}
\newcommand{\F} {{\mathcal F}}
\newcommand{\sing}{{\text{sing}}}
\newcommand{\U} {{\mathcal U}}
\newcommand{\B}{{\mathcal B}}
\newcommand{\X}{{\mathcal X}}

\newcommand{\ECS}[1]{E_{#1}(X)}
\newcommand{\CV}[2]{{\mathcal C}_{#1,#2}(X)}

\newcommand{\rank}{\mathop{\mathrm{rank}}\nolimits}
\newcommand{\codim}{\mathop{\mathrm{codim}}\nolimits}
\newcommand{\Ord}{\mathop{\mathrm{Ord}}\nolimits}
\newcommand{\Var}{\mathop{\mathrm{Var}}\nolimits}
\newcommand{\Ext}{\mathop{\mathrm{Ext}}\nolimits}
\newcommand{\EXT}{\mathop{{\mathcal E}\mathrm{xt}}\nolimits}
\newcommand{\Pic}{\mathop{\mathrm{Pic}}\nolimits}
\newcommand{\Spec}{\mathop{\mathrm{Spec}}\nolimits}
\newcommand{\Jac}{\mathop{\mathrm{Jac}}\nolimits}
\newcommand{\Div}{\mathop{\mathrm{Div}}\nolimits}
\newcommand{\sgn}{\mathop{\mathrm{sgn}}\nolimits}
\newcommand{\supp}{\mathop{\mathrm{supp}}\nolimits}
\newcommand{\Hom}{\mathop{\mathrm{Hom}}\nolimits}
\newcommand{\Sym}{\mathop{\mathrm{Sym}}\nolimits}
\newcommand{\nilrad}{\mathop{\mathrm{nilrad}}\nolimits}
\newcommand{\Ann}{\mathop{\mathrm{Ann}}\nolimits}
\newcommand{\Proj}{\mathop{\mathrm{Proj}}\nolimits}
\newcommand{\mult}{\mathop{\mathrm{mult}}\nolimits}
\newcommand{\Bs}{\mathop{\mathrm{Bs}}\nolimits}
\newcommand{\Span}{\mathop{\mathrm{Span}}\nolimits}
\newcommand{\IM}{\mathop{\mathrm{Im}}\nolimits}
\newcommand{\Hol}{\mathop{\mathrm{Hol}}\nolimits}
\newcommand{\End}{\mathop{\mathrm{End}}\nolimits}
\newcommand{\CH}{\mathop{\mathrm{CH}}\nolimits}
\newcommand{\Exec}{\mathop{\mathrm{Exec}}\nolimits}
\newcommand{\SPAN}{\mathop{\mathrm{span}}\nolimits}
\newcommand{\birat}{\mathop{\mathrm{birat}}\nolimits}
\newcommand{\cl}{\mathop{\mathrm{cl}}\nolimits}
\newcommand{\rat}{\mathop{\mathrm{rat}}\nolimits}
\newcommand{\Bir}{\mathop{\mathrm{Bir}}\nolimits}
\newcommand{\Rat}{\mathop{\mathrm{Rat}}\nolimits}
\newcommand{\aut}{\mathop{\mathrm{aut}}\nolimits}
\newcommand{\Aut}{\mathop{\mathrm{Aut}}\nolimits}
\newcommand{\eff}{\mathop{\mathrm{eff}}\nolimits}
\newcommand{\nef}{\mathop{\mathrm{nef}}\nolimits}
\newcommand{\amp}{\mathop{\mathrm{amp}}\nolimits}
\newcommand{\DIV}{\mathop{\mathrm{Div}}\nolimits}
\newcommand{\Bl}{\mathop{\mathrm{Bl}}\nolimits}
\newcommand{\Cox}{\mathop{\mathrm{Cox}}\nolimits}
\newcommand{\NE}{\mathop{\mathrm{NE}}\nolimits}
\newcommand{\NM}{\mathop{\mathrm{NM}}\nolimits}
\newcommand{\Gal}{\mathop{\mathrm{Gal}}\nolimits}
\newcommand{\coker}{\mathop{\mathrm{coker}}\nolimits}
\newcommand{\ch}{\mathop{\mathrm{ch}}\nolimits}

\title[Xiao's Conjecture]{Xiao's Conjecture on Canonically Fibered Surfaces}

\author{Xi Chen}

\thanks{Research partially supported by Discovery Grant 262265 from the Natural Sciences and Engineering Research Council of Canada.}



\address{632 Central Academic Building\\
University of Alberta\\
Edmonton, Alberta T6G 2G1, CANADA}

\email{xichen@math.ualberta.ca}

\date{June 30, 2017}

\begin{abstract}
A canonically fibered surface is a surface whose canonical series maps it to a curve.
Using Miyaoka-Yau inequality, A. Beauville proved that a canonically fibered surface
has relative genus at most $5$ when its geometric genus is sufficiently large. G. Xiao
further conjectured that the relative genus cannot exceed $4$. We give a proof of this
conjecture.
\end{abstract}

\maketitle

\section{Introduction}\label{SECINTRO}

\subsection{Statement of the Result}

A dominant rational map $f: X\dashrightarrow Y$ between two projective varieties $X$ and $Y$
is a {\em canonical fibration} if it is given by the canonical series $|K_X|$ of
$X$, followed by Stein Factorization; more precisely, $f$ is a dominant rational map with connected fibers
factored through by $X\dashrightarrow \PP H^0(K_X)^\vee$
such that the induced map $Y\dashrightarrow \PP H^0(K_X)^\vee$
is generically finite over its image.

Now let us consider a canonical fibration $f: X\dashrightarrow C$, where
\begin{enumerate}
\item[C1.]
$X$ and $C$ are irreducible, smooth and projective over $\BC$ of dimension
$\dim X = 2$ and $\dim C = 1$, respectively,
\item[C2.]
$f$ has connected fibers, $X$ is of general type and
\item[C3.]
$|K_X| = f^* \D + N$ for a base point free (bpf) linear series $\D$ on $C$ of $\dim \D \ge 1$ and $N$ the fixed part of $|K_X|$. 
\end{enumerate}

We call such $f: X\dashrightarrow C$ a {\em canonically fibered surface}.
Using Miyaoka-Yau inequality (\cite{M} and \cite{Y}), A. Beauville proved that the general fibers of $f$ are curves of genus $\le 5$
when the geometric genus $p_g(X) = h^0(K_X)$ of $X$ is sufficiently large.
There are infinitely many families of canonically fibered surfaces
of relative genus $2$ and $3$ \cite{S}.
On the other hand, G. Xiao conjectured \cite[Problem 6]{X2}

\begin{conj}[Xiao]\label{CONJXIAO}
Let $f: X\dashrightarrow C$ be a canonically fibered surface.
Then $X_p = f^{-1}(p)$ is a curve of genus $\le 4$ for
a general point $p\in C$ when $p_g(X) >> 1$.
\end{conj}

In other words, canonically fibered surfaces of relative genus $5$ are conjectured to have bounded families.
The purpose of this note is to settle the conjecture in the affirmative. More precisely, we will prove

\begin{thm}\label{THMGENUS5}
Let $f: X\dashrightarrow C$ be a canonically fibered surface.
Then the general fibers of $f$ cannot be curves of
genus $5$ if $p_g(X) > 863$ and $C\cong\PP^1$.
\end{thm}

Note that Xiao's conjecture is known for $g(C)\ge 1$.

The paper is organized as follows. In \S \ref{SECBASIC}, we review the known facts about canonically fibered surfaces in general and those of relative genus $5$. Most of these results were due to A. Beauville and G. Xiao and no originality is pretended on our part. In \S \ref{SECOUTLINE}, we give a sketch of our proof of Theorem \ref{THMGENUS5}. Our proof starts in \S \ref{SEC3QUADRICS}, where we construct a {\em pseudo relative canonical model} $Y/C$ of $X/C$.
In \S \ref{SECNC}, we carry out numerical computations on $Y/C$ and reduce our main theorem to an inequality on a certain type of double surface singularities. Finally, we complete our proof by proving this inequality in \S \ref{SECDSS}.

\subsection*{Convention}

We work exclusively over $\BC$ and with analytic topology wherever possible.

\subsection*{Acknowledgments}

I would like to thank Meng Chen and Xiaotao Sun for introducing me to the problem
and Sheng-li Tan and Tong Zhang for some very useful discussions.

\section{Basic Setup}\label{SECBASIC}

Here we review some basic facts on canonically fibered surfaces.

First of all, we have

\begin{thm}[Xiao]\label{THMBASE}
Let $f:X \dashrightarrow C$ be a canonically fibered surface.
Then $C$ is either rational or elliptic if $p_g(X)\ge 3$.
In addition, if $f$ is regular, then
\begin{equation}\label{E522}
f_* \omega_X = f_* \omega_f \otimes \omega_C = L \oplus M
\end{equation}
where $\omega_\bullet$ are the dualizing sheaves,
\begin{itemize}
\item
$L$ is a line bundle satisfying $h^0(L) = p_g(X)$,
\item
$h^0(M) = 0$ and $M\otimes \omega_C^{-1}$ is semi-positive.
\end{itemize}
\end{thm}

More explicitly, $M$ is in the form of
\begin{equation}\label{E520}
M = \CO(a_1) \oplus \CO(a_2) \oplus ... \oplus \CO(a_m)
\end{equation}
for some $-2\le a_1, a_2, ..., a_m\le -1$
when $C\isom \PP^1$ and 
\begin{equation}\label{E555}
M = M_1 \oplus M_2 \oplus ... \oplus M_r
\end{equation}
for semi-stable bundles $M_1, M_2, ..., M_r$ satisfying
$\deg M_k = h^0(M_k) = 0$
when $g(C) = 1$.

We refer the readers to \cite{X1} for the proof of Theorem \ref{THMBASE}, which makes use of Fujita's fundamental result on the semi-positivity of $f_* \omega_f$ \cite{F} (see also \cite{H-V}). 

Next, as mentioned before, using Miyaoka-Yau inequality
\begin{equation}\label{E507}
9 \chi(\CO_X) \ge K_X^2.
\end{equation}
A. Beauville proved \cite{B}

\begin{thm}[Beauville]\label{THMFIBER}
If $f: X\dashrightarrow C$ is a canonically fibered surface with
either $p_g(X) \ge 20$ or $g(C) = 1$,
then $2\le g(F) \le 5$ for a general fiber $F$ of $f$.
\end{thm}

Note that the lower bound $p_g(X)\ge 20$ in Beauville's theorem has been improved to $p_g(X)\ge 12$ in recent years.

After replacing $X$ by one of its birational models,
we may assume that
\begin{enumerate}
\item[C4.]
$f$ is regular and
\item[C5.]
$K_X$ is relatively nef over $C$ and we write $K_X$ as the sum of its moving part and fixed part
\begin{equation}\label{E500}
K_X = f^* D + N
\end{equation}
in $\Pic(X)$, where $D\in \Pic(C)$ is a bpf divisor on $C$ of $\deg D = d$ and
$N$ is an effective divisor on $X$.
\end{enumerate}

Now let us consider canonically fibered surfaces of relative genus $5$. That is, we further assume that
\begin{enumerate}
\item[C6.]
a general fiber $F$ of $f$ is a smooth projective
curve of genus $5$ and 
\item[C7.]
$p_g(X) >> 1$.
\end{enumerate}

Of course, we want to show that there are no canonically fibered surfaces satisfying C1-C7. Assuming such fibration exists, we have

\begin{thm}[Xiao]\label{THMG5}
If $f:X\to C$ is a canonically fibered surface satisfying C1-C6
and either $p_g(X)\ge 85$ or $g(C) = 1$,
then $K_X$ is given by 
\begin{equation}\label{E529}
K_X = f^* D + 8 \Gamma + V,
\end{equation}
where $V$ is effective, $f_* V = 0$ and $\Gamma$ is a section of $f$ satisfying 
\begin{equation}\label{E530}
\begin{aligned}
\boxed{-\frac{d}8 + \frac{17}8 \deg K_C}
&\le \deg K_C - K_X \Gamma = \Gamma^2 = -\frac{d}9 + \frac{\deg K_C}9 -
\frac{\Gamma V}9
\\
&\le -\frac{d}{9} + \frac{\deg K_C}{9}.
\end{aligned}
\end{equation}
\end{thm}

Here we consider $\Gamma V$ as the ``correction term'' in the estimate of $\Gamma^2$, which lies roughly between $-d/8$ and $-d/9$. It follows from \eqref{E530} that $\Gamma V$ is bounded by
\begin{equation}\label{E627}
0\le \Gamma V \le \frac{d}8 - \frac{145}8 \deg K_C.
\end{equation}

\section{Outline of Our Proof}\label{SECOUTLINE}

Suppose that $f: X\to C$ is a canonically fibered surface satisfying C1-C7.
Let us highlight a few key facts about $f$ obtained in the previous section:
\begin{itemize}
\item
$C$ is either rational or elliptic (see Theorem \ref{THMBASE}).
\item
$K_X = f^*D + 8\Gamma + V$, where $8\Gamma +V$ is the fixed part, $D$ is
a bpf divisor on $C$, $f_* V = 0$ and $\Gamma$ is a section
of $f$ (see \eqref{E529}); $\Gamma^2$ has a lower bound given by \eqref{E530}, which implies asymptotically
\begin{equation}\label{E105}
\liminf_{p_g(X)\to \infty} \frac{\Gamma^2}{p_g(X)}\ge -\frac{1}8.
\end{equation}
\item
The pushforward $f_* \omega_f$ of the dualizing sheaf of $f$
splits in a ``highly unbalanced'' way (see \eqref{E522}).
\end{itemize}

Another key fact, due to X.T. Sun \cite{S}, is that a general fiber of $f$ is not hyperelliptic. 
A non-hyperelliptic curve $F$ of genus $5$ has a ``nice'' canonical model:
it is mapped to $\PP^4$ by $|K_F|$ with the image being the complete intersection of three quadrics. Naturally, we may consider the map
\begin{equation}\label{E100}
\kappa: X\dashrightarrow \PP (f_*\omega_f)^\vee = \Proj \left(\oplus \Sym^\bullet f_*\omega_f\right)
\end{equation}
by $\omega_f\otimes f^* G$ for a sufficiently ample line bundle $G$ on $C$,
which simply maps $X/C$ fiberwise to $\PP^4$ by the canonical series.
Sun's theorem tells us that $\kappa$ maps $X$ birationally onto its image
$\kappa(X)$ and we call $\kappa(X)$ a {\em relative canonical model} of
$X$ over $C$.

Indeed, $\PP (f_*\omega_f)^\vee$ in the above construction
can be replaced by $\PP (f_* \LL)^\vee$ for any line bundle $\LL$ on $X$ satisfying
\begin{equation}\label{E103}
\LL\Big|_p \isom \omega_f\Big|_p = \omega_{X_p}
\end{equation}
at a general point $p\in C$. 
It turns out that 
\begin{equation}\label{E104}
\LL = \CO_X(8\Gamma)
\end{equation}
is the best choice of $\LL$ for our purpose,
where $\Gamma$ is the section of $f$ given as above.
We call the image $Y = \tau(X)$ under $\tau: X\dashrightarrow W = \PP(f_*\LL)^\vee$
a {\em pseudo relative canonical model} of $X$
over $C$.

Intuitively, $\tau$ contracts all fibers of $X_p$ that are disjoint from $\Gamma$. And since $\Gamma$ is a section of $f$,
it contracts all components of $X_p$ except the unique component meeting $\Gamma$. As a result, $Y$ might be highly singular. This is the cost we have to pay when we use $\LL$ instead of $\omega_f$ to map $X$. Fortunately, $Y$ is not ``too'' bad:
\begin{itemize}
\item
$\tau$ is regular (see Theorem \ref{THM3QUADRICS});
\item
$Y$ is normal and Gorenstein (see Theorems \ref{THMNORMAL} and \ref{THMDELTAP});
\item
$\tau$ maps $\Gamma$ to a section $\Gamma_Y$ of $Y/C$ and
$\Gamma^2 = \Gamma_Y^2$;
\item
$Y$ has only isolated double surface singularities of type \eqref{E642}.
\end{itemize}
The proof of these statements will be carried out in \S \ref{SEC3QUADRICS}, which forms the technical core of this paper. Since a general fiber $X_p$ of $X/C$ is a non-hyperelliptic curve of genus $5$, its image $Y_p = \tau(X_p)$ is cut out by $3$ quadrics in $W_p\cong \PP^4$. We can find a closed subscheme $Z\subset W$ (defined by \eqref{E553}) such that $Z$ is cut out by ``quadrics'' in $W$ and $Y$ is the only irreducible component of $Z$ that is flat over $C$ (see \eqref{E572}). With $Z$, we can carry out very explicit local analysis of $Y$ and thus prove the above statements on $Y$.

Furthermore, working with $Z$, we are able to derive a formula \eqref{E580} for $\Gamma^2$ using adjunction. This formula involves a local invariant $\delta$ of the singularities of $Y$.

In the last step of our proof, we relate $\delta$ to the moduli invariant
\begin{equation}\label{E107}
12 c_1(f_* \omega_f) - \omega_f^2
\end{equation}
of the fibration $f: X\to C$. If $f$ is semistable, \eqref{E107} counts
the total number of singularities of the fibers. In general, we have to
determine the contribution of each singular fiber of $f$ to \eqref{E107}.
This problem had been extensively studied by S.L. Tan in \cite{T1} and
\cite{T2}. For our purpose, $X$ is a resolution of
the double surface singularities of $Y$ and we just have to figure out
the contribution coming from this resolution.

In the end, we are able to give an upper bound
\eqref{E612} for $\Gamma^2$,
which turns out asymptotically
\begin{equation}\label{E106}
\limsup_{p_g(X)\to \infty} \frac{\Gamma^2}{p_g(X)}\le -\frac{3}{23}
\end{equation}
and consequently contradicts \eqref{E105}. This finishes the proof of Theorem \ref{THMGENUS5}.

\section{Intersections of Three Quadrics}\label{SEC3QUADRICS}

\subsection{A Pseudo Relative Canonical Model of $X/C$}

As outlined in \S \ref{SECOUTLINE}, we consider the rational map
\begin{equation}\label{E552}
\begin{tikzcd}
\tau: X \arrow[dashed]{r} & W = \PP (f_* \LL)^\vee
= \Proj\left(\oplus \Sym^\bullet f_* \LL\right),
\end{tikzcd} 
\end{equation}
where $\LL = \CO_X(8\Gamma)$.

More precisely, the map \eqref{E552} is given
by $|\LL \otimes f^* G|$ for a sufficiently ample line bundle
$G$ on $C$, which first maps $X$ to $\PP H^0(\LL \otimes f^*G)^\vee$.
This is a rational map factoring through
$\PP (f_*\LL)^\vee$ via
\begin{equation}\label{E101}
\begin{tikzcd}
& \PP (f_*\LL)^\vee\arrow{d}\\
X \arrow[dashed]{r} \arrow[dashed]{ur} & \PP H^0(\LL\otimes f^*G)^\vee
\end{tikzcd}
\end{equation}
where the map $\PP (f_* \LL)^\vee
\to \PP H^0(\LL\otimes f^*G)^\vee$ is induced by the
surjection
\begin{equation}\label{E102}
\begin{tikzcd}
H^0(\LL \otimes f^*G) \arrow[two heads]{r} & f_*\LL \Big|_p
\end{tikzcd} 
\end{equation}
for all $p\in C$.

Fiberwise, $\tau$ maps a general fiber $F = X_p$ to $\PP^4 = \PP H^0(K_F)^\vee$ by the canonical map
\begin{equation}\label{E516}
\begin{tikzcd}
F\arrow{r}{|K_F|} & \PP^4.
\end{tikzcd}
\end{equation}

If the general fibers $F$ of
$f$ are hyperelliptic, $\tau$ maps $X$ generically 2-to-1 onto its image.
Fortunately, this possibility has been ruled out by a result of X.T. Sun
\cite{S}:

\begin{thm}[Sun]\label{HYPERELLIPTIC}
If $f:X\to C$ is a canonically fibered surface satisfying C1-C6 and either $p_g(X)\ge 53$ or $g(C) = 1$,
then the general fibers of $f$ are not hyperelliptic.
\end{thm}

Thus, $\tau$ sends $X$ birationally onto its image.
Let us study its image, more precisely,
the proper transform $Y = \tau(X)$ of $X$ under $\tau$. 
We consider $Y/C$ as a pseudo relative canonical model of $X/C$.

The classical Noether's theorem tells us that
the canonical image of a non-hyperelliptic curve $F$ of genus $5$ under
\eqref{E516} is the intersection
of three quadrics. Indeed, we have the exact sequence
\begin{equation}\label{E514}
\begin{tikzcd}
0 \arrow{r} & H^0(I_F(2)) \arrow{r} & \Sym^2 H^0(K_F) \arrow{r} & H^0(2K_F) \arrow{r} & 0
\end{tikzcd}
\end{equation}
where $I_F$ is the ideal sheaf of the image of $F$ in $\PP^4$; it follows that
\begin{equation}\label{E515}
\dim H^0(I_F(2)) = 3.
\end{equation}
Namely, the image $F$ of under the canonical map \eqref{E516}
is cut out by three quadrics in $\PP^4$.

The family version of \eqref{E514} becomes
\begin{equation}\label{E518}
\begin{tikzcd}
0 \arrow{r} & \QQ \arrow{r} & \Sym^2 f_*\LL \arrow{r} & f_* (\LL^{\otimes 2}).
\end{tikzcd}
\end{equation}
Note that \eqref{E518} is only left exact, a priori; the map
$\Sym^2 f_*\LL \xrightarrow{} f_* (\LL^{\otimes 2})$ is generically surjective
by Sun's theorem \ref{HYPERELLIPTIC} and hence its kernel $\QQ$ is a vector bundle of rank $3$
over $C$. In addition, since $f_* (\LL^{\otimes 2})$ is torsion-free, we have an injection
\begin{equation}\label{E519}
\begin{tikzcd}
0 \arrow{r} & \QQ_p = \QQ\Big|_p\arrow{r} & \Sym^2 f_*\LL \Big|_p
\end{tikzcd}
\end{equation}
over every point $p\in C$. Thus $\QQ_p$ can be regarded as a subspace of $H^0(I_F(2))$ of dimension
$3$ for every fiber $F = X_p$. In other words, $\QQ_p$ is a net ($2$-dimensional linear series)
of quadrics passing through $Y_p$. Note that $Y_p$ is contained
in the base locus $\Bs(\QQ_p)$ of $\QQ_p$ but it is not necessarily
true that
\begin{equation}\label{E592}
Y_p = \Bs(\QQ_p)
\end{equation}
for every $p\in C$; \eqref{E592} holds, a priori,
only for a general point $p\in C$.

We may regard $\QQ$ as a subbundle of $\pi_* \CO_W(2) \isom \Sym^2 f_* \LL$, where
\begin{equation}
\begin{tikzcd}
\pi: W = \PP (f_*\LL)^\vee\arrow{r} & C
\end{tikzcd}
\end{equation}
is the projection and $\CO_W(1)$ is the tautological
bundle. Thus, $H^0(\QQ \otimes G)$ is a subspace of $H^0(\CO_W(2) \otimes \pi^*G)$
for $G\in \Pic(C)$. 

We abuse notation a little by using $\Bs (\QQ \otimes G)$ to denote the base locus of $H^0(\QQ \otimes G)$ as a subspace
of $H^0(\CO_W(2) \otimes \pi^*G)$
and defining
\begin{equation}\label{E553}
Z = \bigcap_{G\in \Pic(C)} \Bs(\QQ \otimes G) \subset W.
\end{equation}
Namely, $Z$ is the closed subscheme of $W$ whose fiber $Z_p$ is the intersection of the three quadrics generating $\QQ_p$ over each $p\in C$.
It is also obvious that $Z = \Bs(\QQ\otimes G)$ for a sufficiently ample line bundle $G$ on $C$.

Clearly, $Y\subset Z$.
Indeed, it is easy to see that $Y$ is the only irreducible component of $Z$ that is flat over $C$; the other irreducible components of $Z$, if any, must be
supported in some fibers $W_p$ of $W$ over $C$. However, it is not clear whether $\dim Z = 2$ and whether $Z$ has components other than $Y$.
The following theorem answers these questions.

\begin{thm}\label{THM3QUADRICS}
Let $f:X\to C$ be a flat family of curves of genus $5$ over a smooth
curve $C$ with $X$ smooth and $X_p = f^{-1}(p)$ non-hyperelliptic for $p\in C$ general and
let $\LL$ be an effective line bundle on $X$ with $\Lambda\in |\LL|$ satisfying
\begin{itemize}
 \item \eqref{E103} holds for $p\in C$ general,
 \item $\Lambda\cap X_p$ is a finite set of points
lying on a unique component $B_p$ of $X_p$ for each $p\in C$, and
\item $X_p$ is smooth at each point in $\Lambda\cap X_p$ for all $p\in C$.
\end{itemize}
Let
$\tau$, $W$, $\QQ$, $Y$ and $Z$ be defined as above. Then
\begin{itemize}
\item
$\tau$ is regular and $\tau_* E = 0$ for all components $E\ne B_p$ of $X_p$
and all $p\in C$.
\item
$Z$ is a local complete intersection (l.c.i) of pure dimension $2$.
\item
$Y$ is the only irreducible component of $Z$ that is flat over $C$.
\item
Every quadric of $\QQ_p$ has rank $\ge 3$ for all $p\in C$,
where the rank of a quadric in $\PP^N$ is the rank of its defining equation
as a quadratic form.
\item
$\dim(Z\cap W_p) = 2$, i.e., $Z$ has an irreducible component
contained in $W_p$ if and only if $\Bs(\QQ_p)$ is a cone over a smooth rational
cubic curve in $\PP^4$.
\item
For each $p\in C$,
$\tau$ maps the component $B_p$ either birationally
onto a curve of degree $8$ and arithmetic genus $5$
or $2$-to-$1$ onto a rational normal curve in $\PP^4$.
\item 
$\widehat{Y}_p$ is integral (reduced and irreducible) for all $p\in C$
and $Y_p\isom \widehat{Y}_p$ if and only if $\tau$ maps $B_p$ birationally
onto its image $\tau(B_p)$,
where $\nu: \widehat{Y}\to Y\subset W$ is the normalization of $Y$.
\item
The map $\nu^{-1}\circ \tau: X\to \widehat{Y}$ is a local isomorphism
in an open neighborhood of $\Lambda$.
\end{itemize}
\end{thm}

For starters, it is obvious that $\tau$ is regular outside of $\Lambda$ and
it contracts all curves $E\subset X$ satisfying $E . \Lambda = 0$ and hence
contracts all components of $X_p$ other than $B_p$.
To show that $\tau$ is regular everywhere,
we make two key observations regarding $\tau(B_p)$:
\begin{equation}\label{E545}
\tau(B_p)\subset \PP^4 \text{ is non-degenerate, i.e., }
\tau(B_p) \not\subset P\isom \PP^3\subset \PP^4
\end{equation}
for all hyperplanes $P$ of $\PP^4$ and
\begin{equation}\label{E546}
p_a(\tau(B_p)) \ge 5 \text{ if $\tau$ maps $B_p$ birationally
onto $\tau(B_p)$}
\end{equation}
where $p_a(A)$ is the arithmetic genus of a curve $A$.
The first observation \eqref{E545}
is a consequence of the following lemma.

\subsection{A Lemma on Fibrations of Curves}

\begin{lem}\label{LEMFOC}
Let $f: X\to C$ be a flat projective morphism with connected fibers from
a smooth variety $X$ to a smooth curve $C$. Let
$D, E$ and $G$ be effective divisors on $X$ such that
\begin{itemize}
\item
$\supp(D) \not\subset \supp(E)$,
\item
$D + E = F$ for a fiber $F = f^{-1}(p)$ of $f$ and
\item
$E\cap G = \emptyset$.
\end{itemize}
Then $H^0(\CO_X(G - D)) = H^0(\CO_X(G-F))$.  
\end{lem}

It was suggested to the author by the referee that the above lemma follows easily from the following one:

\begin{lem}\label{LEMCHENMENG}
Let $X$ be a smooth projective surface and $E$ be an effective divisor supported on irreducible components with negative intersection matrix. For a divisor $A$ on $X$, if $A\Gamma = 0$ for all irreducible components $\Gamma$ of $E$, then $H^0(A) = H^0(A+E)$.
\end{lem}

\begin{proof}
We argue by induction on the number of components in $E$ counted with multiplicity. There is nothing to prove if $H^0(A + E) = 0$. Suppose that $A+E$ is effective. We write $|A + E| = Z + |M|$ with $Z$ the fixed part and $M$ the moving part of $|A + E|$. Since $ZE = (A+E)E - ME = E^2 - ME < 0$, $Z$ and $E$ have some common components. Suppose that $\Gamma$ is a common irreducible component of $Z$ and $E$. Then $H^0(A+E) = H^0(A+(E-\Gamma))$ and we are done by induction.
\end{proof}

Then Lemma \ref{LEMFOC} follows easily from \ref{LEMCHENMENG}: first we cut down $X$ by sufficiently ample divisors to $\dim X=2$ and then apply \ref{LEMCHENMENG}
by setting $A = G-F$.

\subsection{Proof of Theorem \ref{THM3QUADRICS}}

Let us see how to derive \eqref{E545} from Lemma \ref{LEMFOC}.
By the lemma, we have
\begin{equation}\label{E547}
\begin{split}
H^0(\LL\otimes f^*G \otimes \CO_X(-D))
&= H^0(\LL\otimes f^*G \otimes \CO_X(-F))\\
& \isom H^0(f_* \LL \otimes G(-p))
\end{split}
\end{equation}
where $D = B_p$ and $F = X_p$. Therefore, we have the diagram
\begin{equation}\label{E548}
\begin{tikzcd}[column sep=small]
0 \arrow{r} & H^0(f_*\LL \otimes G(-p)) \arrow[equal]{d}\arrow{r} & H^0(f_* \LL \otimes G) \arrow[equal]{d}\arrow{r} & f_*\LL \Big|_p \arrow{r}\arrow{d} & 0\\
0 \arrow{r} & H^0(\LL(-D)\otimes f^* G) \arrow{r} & H^0(\LL \otimes f^* G)
\arrow{r} & H^0(D, \LL)
\end{tikzcd}
\end{equation}
from which we conclude that the map
\begin{equation}\label{E549}
\begin{tikzcd}
f_*\LL \Big|_p\arrow[hook]{r} & H^0(D,\LL)
\end{tikzcd}
\end{equation}
is an injection. It follows that $\tau(D) = \tau(B_p)$ is a non-degenerate curve
in $\PP^4$. And since $\tau(B_p)$ is irreducible,
$\QQ_p$ does not contain a quadric of rank $\le 2$, i.e.,
the union of two hyperplanes.
This implies that every member of $\QQ_p$ has rank $\ge 3$,
and a general member of $\QQ_p$ has rank $\ge 4$.
It follows that $\dim \Bs(\QQ_p) \le 2$ for all $p\in C$. 
At a general point $p\in C$ where $X_p$ is a smooth non-hyperelliptic curve,
we know that $X_p \isom Y_p = Z_p = \Bs(\QQ_p)$ has dimension $1$.
Consequently, $Z$
has pure dimension $2$, $Y$ is the only component of $Z$ flat over $C$
and $Z$ has a component supported in $W_p$ precisely at the points $p$ where
$\dim \Bs(\QQ_p)$ ``jumps''. The following lemma tells us exactly how three
linearly independent quadrics fail to meet properly.

\begin{lem}\label{LEM729000}
Let $\QQ\subset \PP H^0(\CO_{\PP^n}(2))$
be a net of quadrics in $\PP^n$.
If the base locus $\Bs(\QQ)$ of $\QQ$ has dimension $> n - 3$, i.e., the three quadrics
generating $\QQ$ fail to meet properly,
then one of the following holds
\begin{itemize}
\item every quadric $Q\in \QQ$ has rank $\le 4$ or
\item $\QQ$ contains a pencil of quadrics of rank $\le 2$.
\end{itemize}
\end{lem}

\begin{proof}
Suppose that a general quadric of $\QQ$ has rank $\ge 5$ and
$\QQ$ is generated by $Q, Q'$ and $Q''$ with $\rank(Q) \ge 5$.
Since $\Pic(Q) = \BZ$, $Q\cap Q' \cap Q''$ has dimension $>\dim Q - 2$ if and only
if $\Lambda \subset Q'\cap Q''$ for some hypersurface $\Lambda \subset Q$. Clearly,
$\Lambda$ is a hyperplane section, $Q' = \Lambda \cup \Lambda'$
and $Q'' = \Lambda \cup \Lambda''$ on $Q$. It follows that $\QQ$ contains a pencil
of quadrics of rank $\le 2$.
\end{proof}

By the lemma, we see that $\dim \Bs(\QQ_p) = 2$ only if
either every quadrics of $\QQ_p$ has rank $\le 4$ or $\QQ_p$ has
a quadric of rank $\le 2$. We have eliminated the latter by \eqref{E545}. This leaves
us the former: every quadric of $\QQ_p$ is a cone of a common vertex
over a quadric in some $P\isom \PP^3\subset \PP^4$.
It is not hard to see further that $\Bs(\QQ_p)$ is a cone over
a rational normal curve, i.e., a smooth rational cubic curve in $P$.

Since $\tau(B_p)$ is non-degenerate in $\PP^4$, it has degree at least $4$.
Therefore, $\tau$ maps $B_p$ either birationally or $2$-to-$1$ onto to $\tau(B_p)$. If it is the latter, $\tau(B_p)$ must have degree $4$ and hence
it is a rational normal curve in $\PP^4$.

Suppose that $\tau$ maps $B_p$ birationally onto
$J = \tau(B_p)\subset \PP^4$. Let us first justify \eqref{E546}, i.e.,
show that 
\begin{equation}\label{E560}
p_a(J) \ge 5.
\end{equation}
This follows from the following lemma.

\begin{lem}\label{LEMGENUS}
Let $X$ and $Y$ be two surfaces proper and flat over the unit disk
$\Delta = \{|t| < 1\}$ with $X_t$ and $Y_t$ smooth and irreducible
for $t\ne 0$.
Let $\tau: X\dashrightarrow Y$ be a birational map with the diagram
\begin{equation}\label{E561}
\begin{tikzcd}
X \arrow[dashed]{r}{\tau}\arrow{d} & Y\arrow{dl}\\
\Delta
\end{tikzcd}
\end{equation}
such that
\begin{itemize}
\item
$\tau$ is regular on $X\backslash \Sigma$ for a finite set $\Sigma$ of points on $X_0$,
\item
$X_0$ is smooth at each point $x\in \Sigma$,
\item
$\tau_* I \ne 0$ for each component $I$ of $X_0$ satisfying $I\cap \Sigma\ne \emptyset$
and
\item
$J = \tau_* X_0$ is reduced.
\end{itemize}
Then $p_a(J) \ge g(Y_t)$ for $t\ne 0$.
\end{lem}

\begin{proof}
WLOG, we may assume that both $X$ and $Y$ are normal.

Let $E\subset Y_0$ be the exceptional locus of $\tau^{-1}$ over $\Sigma$. More
precisely, $E = \pi_Y(\pi_X^{-1}(\Sigma)\cap G_\tau)$ with $G_\tau\subset X\times
Y$ the graph of $\tau$ and $\pi_X$ and $\pi_Y$ the projections of
$X\times Y$ to $X$ and $Y$. It is not hard to see by Zariski's Main Theorem
that $\tau: X\backslash \Sigma\to Y\backslash E$ is proper.

After a base change, we can find
$\widehat{X}$ and $\widehat{Y}$ with the diagram
\begin{equation}\label{E559}
\begin{tikzcd}
\widehat{X} \arrow{d}[left]{f} \arrow{r}{\widehat{\tau}} & \widehat{Y}\arrow{d}[right]{g}\\
X \arrow[dashed]{r}{\tau} & Y
\end{tikzcd}
\end{equation}
and the properties that
\begin{itemize}
\item
$f, g$ and $\widehat{\tau}$ are birational and proper maps inducing isomorphisms
$X\backslash f^{-1}(\Sigma)\isom X\backslash \Sigma$ and
$Y\backslash g^{-1}(E)\isom Y\backslash E$
and
\item
$f^{-1}(U_\Sigma)\cap \widehat{X}_0$ and
$g^{-1}(V_E)\cap \widehat{Y}_0$ have simple normal crossing
for analytic open neighborhoods $U_\Sigma\subset X$ and
$V_E\subset Y$ of $\Sigma$ and $E$, respectively.
\end{itemize}
Since $X_0$ is smooth at each $x\in \Sigma$, $f^{-1}(x)$ must be a tree of smooth rational
curves meeting $\widehat{I}$
transversely at one point, where $\widehat{I}$ is the proper transform of
the component $I$ containing $x$ under $f$.
Consequently,
$g^{-1}(E) = \widehat{\tau}(f^{-1}(\Sigma))$ is a disjoint union of trees of
smooth rational curves, each meeting $\widehat{J}$ transversely at one point,
where $\widehat{J}$ is the proper transform of $J$ under $g$.
Therefore, $p_a(J) \ge p_a(\widehat{J}) = p_a(\widehat{Y_0}) = g(Y_t)$.
\end{proof}

This proves \eqref{E560}.
Let us consider the Hilbert polynomial
\begin{equation}\label{E556}
\chi(\CO_J(n)) = h^0(\CO_J(n)) - h^1(\CO_J(n)) = \delta n - p_a(J) + 1
\end{equation}
of $J\subset \PP^4$, where $\delta$ is the degree of $J$ in $\PP^4$.

Let $R = \Bs(\QQ_p)$. If $\dim R = 1$, then
\begin{equation}\label{E557}
h^1(\CO_J(1))\le h^1(\CO_R(1)) = 1
\end{equation}
and hence
\begin{equation}\label{E558}
5 \le h^0(\CO_J(1)) = \delta -p_a(J) + 1 + h^1(\CO_J(1)) \le \delta - 3 
\end{equation}
by \eqref{E560}; it follows that $\delta \ge 8$.

Suppose that $\dim R = 2$. We have proved that $R$ is a cone over a smooth
rational cubic curve. Let $\nu: \BF_3 \isom \widehat{R} \to R$ be the
blowup of $R$ at the vertex of the cone. Then
\begin{equation}\label{E562}
\begin{split}
\chi(\CO_J(n)) &= h^0(\CO_J(n)) = h^0(\CO_R(n)) - h^0(\CO_R(n)\otimes \I_J)\\
&=h^0(\CO_{\widehat{R}}(n)) - h^0(\CO_{\widehat{R}}(n)\otimes 
\CO_{\widehat{R}}(-\widehat{J}))\\
&= \delta n - \left(\left\lceil \frac{\delta}3 \right \rceil - 1\right)
\left(\delta - \frac{3}2 \left\lceil \frac{\delta}3 \right\rceil - 1\right)+1
\end{split}
\end{equation}
for $n>>1$ and hence
\begin{equation}\label{E563}
p_a(J) =
\left(\left\lceil \frac{\delta}3 \right \rceil - 1\right)
\left(\delta - \frac{3}2 \left\lceil \frac{\delta}3 \right\rceil - 1\right),
\end{equation}
where $\I_J$ is the ideal sheaf of $J$ in $R$,
$\CO_{\widehat{R}}(1) = \nu^*\CO_R(1)$ is the pullback of the hyperplane bundle
and $\widehat{J}\subset \widehat{R}$ is the proper transform of $J$.
Therefore, we must have $\delta \ge 8$ by \eqref{E560} again.

Therefore, $\deg J = 8$ when $\tau$ maps $B_p$ birationally onto $J$.
We also proved in the above argument that $p_a(J) = 5$.

In conclusion, $\tau$ maps $B_p$ either birationally onto a curve of degree
$8$ and arithmetic genus $5$ or $2$-to-$1$ onto a rational normal curve
of degree $4$. 
This implies that $\tau(B_p)$ is the only component of dimension $1$ in $Y_p$ and hence $\tau$ is regular along $X_p$.
Therefore, $\tau$ is regular everywhere.

When $\tau$ maps $B_p$ birationally onto a curve $J$. Since $\deg J = 8$ and
$p_a(J) = 5$, we necessarily have $\widehat{Y}_p \isom J = Y_p$. When
$\tau$ maps $B_p$ $2$-to-$1$ onto a rational normal curve $J$,
since $\tau: X\to Y$ factors through $\widehat{Y}$,
$\widehat{Y}_p$ is birational to $B_p$ and $\nu$ maps $\widehat{Y}_p$ $2$-to-$1$
onto $J$ by Zariski's Main Theorem.
Therefore, $\widehat{Y}_p$ is integral for all $p$.

Since $X_p$ is smooth at $\Lambda\cap X_p$, $X_p\isom \widehat{Y}_p$
for $p\in C$ general and $\widehat{Y}_p$ is reduced for all $p\in C$,
we conclude that $\nu^{-1}\circ \tau$ is a local isomorphism along $\Lambda$.

This finishes the proof of Theorem \ref{THM3QUADRICS}.

\subsection{Normality of $Y$}

So far we have only made use of the fact that there exists $\Lambda\in |\LL|$ which
is a multi-section of $f$ and meets each fiber at the smooth points of
a unique component of that fiber. We can certainly say more about $Y$
when $\Lambda = 8\Gamma$ with $\Gamma$ a section of $f$.

\begin{thm}\label{THMNORMAL}
In addition to the hypotheses of Theorem \ref{THM3QUADRICS}, we further assume that
$\Lambda = 8\Gamma$ with $\Gamma$ a section of $f$. Then
$Y$ is normal.
\end{thm}

By Theorem \ref{THM3QUADRICS},
$Y$ being normal is equivalent to saying that $\tau$ maps
$B_p$ birationally onto $Y_p$ for every $p\in C$. Namely, it excludes the cases
that $\tau$ maps $B_p$ $2$-to-$1$ onto a rational normal curve in $\PP^4$.
The implication of this statement is somewhat surprising: the family $X/C$ does not have any
hyperelliptic fibers. One can compare this with Sun's theorem which says that
the general fibers of $X/C$ are not hyperelliptic.
If we let $\m_5$ be the moduli space of curves of genus $5$ and $\U$ be the subvariety
of $\m_5$ consisting of curves $F$ with the property that $K_F = \CO_F(8q)$ for some
$q\in F$, then $\U$ has codimension $3$ in $\m_5$ and contains
the hyperelliptic locus $\U_e$ as an ``isolated'' component
in the sense that $\U_e$ is disjoint from the other components
of $\U$.

We abuse the notation a little by using $\Lambda$ to also denote
the member of $|\CO_W(1)|$ corresponding to $8\Gamma\in |\LL|$.
On each fiber $W_p$, $\Lambda$ is a hyperplane
meeting $Y_p$ properly at a unique point with multiplicity $8$
by Theorem \ref{THM3QUADRICS}. Likewise, $\nu^{-1}(\Lambda)$ meets each fiber
$\widehat{Y}_p$ at a unique point. This results in two sections $\Gamma_Y$
and $\widehat{\Gamma}_Y$ of $Y/C$ and $\widehat{Y}/C$, respectively, with the
properties that 
\begin{equation}\label{E599}
\widehat{\Gamma}_Y = \nu^{-1}(\Gamma_Y) \text{ and }
\Gamma = \tau^{-1}(\Gamma_Y).
\end{equation}

Suppose that $\tau$ maps $B_p$ $2$-to-$1$ onto its image, i.e., $\nu$ maps
$\widehat{Y}_p$ $2$-to-$1$ onto its image for some $p\in C$. Let
$\widehat{q} = \widehat{\Gamma}_Y \cap \widehat{Y}_p$
and $q = \Gamma_Y\cap Y_p$.
Also by Theorem \ref{THM3QUADRICS}, $\widehat{Y}_p$ is smooth at $\widehat{q}$.
And since $\nu^{-1}(q) = \{\widehat{q}\}$, $\nu$ is ramified at $\widehat{q}$.

Let $J = \tau_* B_p = \nu_* \widehat{Y}_p$ be the scheme-theoretic image of $B_p$ and $\widehat{Y}_p$ under $\tau$ and $\nu$, respectively. Then $J$ is a closed subscheme of pure dimension $1$ in the cone $R = \Bs(\QQ_p)$ with multiplicity $2$ along a smooth rational curve of degree $4$ on $R$. The Hilbert polynomial of $J\subset \PP^4$ can be computed in the same way as \eqref{E562} with $\delta = \deg J = 8$. Therefore, $p_a(J) = p_a(Y_t) = 5$ for all $t\in C$. In particular, $p_a(J) = p_a(Y_p)$. Hence $Y_p$ 
does not contain any embedded points and $Y_p = J$. This implies that $Y_p$ has pure dimension $1$ and $Y$ is Cohen-Macaulay along $Y_p$. 

In summary, we have the following facts regarding $Y$ and $\widehat{Y}$
locally at $q$ and $\widehat{q}$:
\begin{itemize}
 \item $\nu$ maps $\widehat{Y}_p$ $2$-to-$1$ onto $\supp(Y_p)$, which is a smooth
 curve;
 \item $\Lambda$ is a Cartier divisor on $Y$ such that
$\Lambda = 8\Gamma_Y$ and $\nu^*\Lambda = 8 \widehat{\Gamma}_Y$
with $\Gamma_Y$ and $\widehat{\Gamma}_Y$ sections of $Y/C$ and $\widehat{Y}/C$,
respectively;
\item $Y$ is Cohen-Macaulay at $q$;
\item $\widehat{Y}_p$ is smooth at $\widehat{q} = \widehat{\Gamma}_Y\cap
\widehat{Y}_p$;
\item $\nu$ is ramified at $\widehat{q}$.
\end{itemize}

It is not hard to see from the above that $Y$ is locally given by
\begin{equation}\label{E600}
\{ y^2 = t^{2m}x \} \subset \Delta_{xyt}^3 
\end{equation}
for some $m\in \BZ^+$ at $q$ with $\Lambda$ cut out by
\begin{equation}\label{E601}
y - \lambda(x, t) = 0
\end{equation}
for some $\lambda(x,t) \in \BC[[x,t]]$ satisfying $\lambda(x,0) = x^4$.
And since $\Lambda = 8\Gamma_Y$ on $Y$, we necessarily have
\begin{equation}
(\lambda(x,t))^2 - t^{2m}x = (x+g(x,t))^8
\end{equation}
for some $g(x, t)\in \BC[[x, t]]$ satisfying $g(x, 0) \equiv 0$,
which is impossible.

This shows that $Y$ is normal.

\subsection{Multiplicities of Cubic Cones in $Z$}

We have shown that $Z\subset W$ is the union of $Y$ and cones over smooth rational cubic
curves, which we call {\em cubic cones}, contained in the fibers $W_p$. Therefore,
\begin{equation}\label{E572}
Z = Y + \sum_{p\in C} \delta_p S_p,
\end{equation}
where $S_p$ is a cubic cone contained in $W_p$ and $\delta_p$ is its multiplicity
in $Z$ (let $\delta_p = 0$ and $S_p = \emptyset$ if such $S_p$ does not exist). 
The multiplicity $\delta_p$ has a nice algebraic interpretation.

\begin{defn}\label{DEFDISCR}
Let $E$ be a vector bundle of rank $n$ over a variety $C$ and $W = \PP E^\vee$. 
For $s\in H^0(\CO_W(2)\otimes \pi^* L)$, the {\em discriminant locus} $D_s\subset C$
of $s$ is the vanishing locus of the determinant $\det(s)$ of $s$, 
considered as a section in
\begin{equation}\label{E603}
H^0(\Sym^2 E \otimes \pi^*L)
\subset \Hom(E^\vee, E\otimes \pi^*L), 
\end{equation}
where $\pi$ is the projection $W\to C$ and $L$ is a line bundle on $C$.
Obviously, $D_s\in |2c_1(E) + nL|$.

Locally $D_s$ is very easy to describe. Suppose that $C\cong \Delta^N$ is a polydisk. Then $W\cong \PP^{n-1}\times \Delta^N$ and
\begin{equation}\label{E908}
s = \sum_{ij} f_{ij}(t) w_i w_j
\end{equation}
for some $f_{ij}\in \BC[[t]]$ under the homogeneous coordinates $(w_0, w_1, ..., w_{n-1})$ of $\PP^{n-1}$ and the coordinates $t = (t_1,t_2,...,t_N)$ of $\Delta^N$. Then the discriminant locus of
$s$ is the vanishing locus of $\det\begin{bmatrix}
f_{ij}(t)
\end{bmatrix}$, as a subscheme of $\Delta^N$, which is exactly the locus of $t\in \Delta^N$ where the quadric $s(t) = 0$ is singular.
\end{defn}

Our purpose is twofold: first, we will show that $\delta_p$ is the multiplicity
of the discriminant locus of a general member of $\QQ$ at $p$;
second, we will describe the type
of singularity $Y$ has at the vertex of each cubic cone $S_p$, which depends on
$\delta_p$ as we will see.

\begin{thm}\label{THMDELTAP}
Under the same hypotheses of Theorem \ref{THMNORMAL}, 
\begin{itemize}
\item
$\delta_p$ is determined by
\begin{equation}\label{E604}
\delta_p = \min \{ \mult_p(D_s): s\in H^0(\QQ\otimes \pi^* G), G\in\Pic(C)\}
\end{equation}
for all $p\in C$,
where $D_s\subset C$ is the discriminant locus of $s$ defined as above and
$\mult_p(D_s)$ is the multiplicity of $D_s$ at $p$.
\item
$Y$ is a l.c.i and hence Gorenstein.
\item
If $S_p \ne \emptyset$, $Y$ has a singularity at the vertex of the cone $S_p$
of type $\{y^2 = g(x,t)\}\subset\Delta_{xyt}^3$ with $g(x,t)\in \BC[[x,t]]$ satisfying
\begin{equation}\label{E642}
g(x, 0) = x^n,\ g(0, t) = t^{\delta_p} \text{ and }
\left.\frac{\partial g}{\partial x}\right|_{x = 0} \equiv 0. 
\end{equation}
for some $n\ge 2$, where $t$ is the local parameter of $C$ at $p$.
\item
If $S_p\ne \emptyset$ and $g(Y_p) = 0$, then $n\le 6$ in \eqref{E642}.
\item
For each $p\in C$, $Y_p$ is an integral curve of degree $8$ and arithmetic genus $5$ with the surjective map
\begin{equation}\label{E927}
\begin{tikzcd}
\Sym^2 H^0(\CO_{Y_p}(1)) \ar[two heads]{r} \ar[equal]{d}
& H^0(\CO_{Y_p}(2)) \ar[equal]{d}\\
\BC^{15} \ar[two heads]{r} & \BC^{12}
\end{tikzcd}
\end{equation}
and hence the left exact sequence \eqref{E518} is exact:
\begin{equation}\label{E931}
\begin{tikzcd}
0 \arrow{r} & \QQ \arrow{r} & \Sym^2 f_*\LL \arrow{r} & f_* (\LL^{\otimes 2})
\arrow{r} & 0.
\end{tikzcd}
\end{equation}
\end{itemize}
\end{thm}

The problem is obviously local: we just have to study $W, X, Y, Z$ and $\QQ$ over
an open neighborhood of $p\in C$. Thus, it suffices to prove the
following:

\begin{lem}\label{LEMDELTAP}
Let $W = \PP^4 \times \Delta$, $\QQ$ be a subbundle of $\pi_* \CO_W(2)$
of rank $3$ and $Z = \Bs(\QQ)$, where $\Delta = \{|t| < 1\}$,
$\pi$ is the projection $W\to \Delta$, $\CO_W(1)$ is the pullback
of $\CO_{\PP^4}(1)$ and $\Bs(\QQ)$ is the base locus of $H^0(\QQ)$ as
a subspace of $H^0(\CO_W(2))$. Suppose that
\begin{itemize}
\item $Z_t = \Bs(\QQ_t)$ is a smooth curve for $t\ne 0$;
\item $Z = Y + \delta S$, where $S = \Bs(\QQ_0)$ is a cubic cone, $\delta$
is the multiplicity of $S$ in $Z$ and $Y$ is the irreducible component of $Z$ flat
over $\Delta$;
\item $Y_0$ is not contained in a union of lines on $S$ passing through its vertex.
\end{itemize}
Then
\begin{itemize}
\item $\delta$ is determined by
\begin{equation}\label{E650}
\delta = \min \{ \mult_0(D_s): s\in H^0(\QQ) \}.
\end{equation}
\item
$Y$ is a l.c.i and hence Gorenstein.
\item
$Y$ is locally isomorphic to
the surface $\{y^2 = g(x,t)\}\subset\Delta_{xyt}^3$ for some $g(x,t)\in \BC[[x,t]]$ satisfying
\begin{equation}\label{E644}
g(0, t) = t^{\delta} \text{ and }
\left.\frac{\partial g}{\partial x}\right|_{x = 0} \equiv 0,
\end{equation}
at the vertex $q$ of the cone $S$, where $g(x, 0) \not \equiv 0$ if $Y_0$ is reduced at $q$.
\item If $Y_0$ is integral, it is a curve of degree $8$ and arithmetic genus $5$ in $\PP^4$ with the surjective map
\begin{equation}\label{E932}
\begin{tikzcd}
\Sym^2 H^0(\CO_{Y_0}(1)) \ar[two heads]{r} \ar[equal]{d}
& H^0(\CO_{Y_0}(2)) \ar[equal]{d}\\
\BC^{15} \ar[two heads]{r} & \BC^{12}
\end{tikzcd}
\end{equation}
\end{itemize}
\end{lem}

\begin{proof}
Let $Q_1$ and $Q_2$ be two general members of $H^0(\QQ)$ and let $V = Q_1\cap Q_2$.
It is easy to see that $V_0 = S\cup T$, where $T\isom \PP^2$ is a $2$-plane
meeting $S$ along two distinct lines passing through the vertex $q$ of $S$.
Since $V_t$ is smooth for $t\ne 0$ and $S$ and $T$ meet transversely outside of
$q$, $V$ is a $3$-fold locally given by $xy = t^m$ in
$\Delta_{xyzt}^4$ at a general point of $S\cap T$, where $m\in \BZ^+$ is a constant
on each line in $S\cap T$. As $Q_1$ and $Q_2$ vary, $T$ varies and the
corresponding monodromy action on the two lines in $S\cap T$ is transitive. Therefore,
$m$ remains constant at general points of $S\cap T$.
Thus, $V$ is $\BQ$-factorial outside of
finitely many points on $S\cap T$. More precisely, there exists $\Sigma\subset S\cap T$
such that $\dim \Sigma = 0$ and $\CH^1(V\backslash \Sigma)$ is generated by
$\Lambda$ and $S$, where $\Lambda$ is the hyperplane divisor and
$a \Lambda + b S$ is Cartier on $V\backslash \Sigma$ if and only if
$a\in\BZ$ and $b\in m \BZ$.

Clearly, $Z \sim_\text{rat} 2\Lambda$ and hence
\begin{equation}\label{E652}
Y = 2\Lambda - \delta S
\end{equation}
in $\CH^1(V\backslash \Sigma)$. Note that $Y$ is, a priori,
a Weil divisor on $V$. Restricting $Y$ to $T$, we have
\begin{equation}\label{E653}
Y\Big|_T = 2\Lambda - \delta S\Big|_T = \left(2 - \frac{2\delta}{m}\right) \Lambda
\end{equation}
in $\CH_\BQ^1(T\backslash \Sigma)$. And since $Y_0\subset S$, $Y$ meets $T$
at finitely many points and hence $Y.T = 0$. Indeed, although we do not need it,
it is not hard to see that $\Sigma$ is precisely $Y\cap T$. In any event, we can
conclude that $m = \delta$ by \eqref{E653}.

In addition, $Y$ is Cartier on $V\backslash\Sigma$
by \eqref{E652} since $m | \delta$. Therefore, $Y\backslash \Sigma$ is a l.c.i.
Recall that $Q_1$ and $Q_2$ are two general members of $H^0(\QQ)$ and $\Sigma$
is a set of finitely many points on $S\cap T$. As $Q_1$ and $Q_2$ vary, $T$ varies
and it cuts out on $S$ a linear system of two lines $S\cap T$ with the only base point $q$.
Therefore, we conclude that $Y$ is a l.c.i outside of the vertex $q$.

We let $\mu$ be the number defined by the right hand side of \eqref{E650}. A priori, we do not know that $\mu = \delta$, which is what we are going to prove next.

Let $W$ be parameterized by 
$(w_0, w_1, w_2, w_3, w_4, t)$
with $(w_0, w_1, w_2, w_3, w_4)$ the homogeneous coordinates of $\PP^4$.
Obviously, we may choose the vertex $q$ of $S$ to be $(1,0,0,0,0,0)$.
Furthermore, we claim that for a suitable choice of coordinates,
or equivalently, after applying
a suitable automorphism in $\Aut(W/\Delta)$ to $W$, every
$s\in H^0(\QQ)$ is given by $s= s(w_0, w_1, w_2, w_3, w_4,t)$ satisfying
\begin{equation}\label{E649}
\left.\frac{\partial^{j+1} s}{\partial w_0\partial t^j}\right|_{t=0} \equiv 0
\text{ for } j=0,1,2,...,\mu-1.
\end{equation}
To see this, let us consider $\QQ\otimes \BC[t]/(t^\mu)$. By the definition of
$\mu$, every quadric in $H^0(\QQ) \otimes \BC[t]/(t^\mu)$ has rank $\le 4$
over $\BC[t]/(t^\mu)$. And since a general quadric in $H^0(\QQ_0)$ has a unique
singularity at $q$, we conclude that every quadric in
$H^0(\QQ) \otimes \BC[t]/(t^\mu)$ has a singularity at
$(1, q_1(t), q_2(t), q_3(t), q_4(t))$ for some $q_k(t)\in \BC[t]$ satisfying
$q_k(0) = 0$. Then applying the automorphism in $\Aut(W/\Delta)$ sending
\begin{equation}\label{E651}
\begin{split}
&\quad (w_0, w_1 , w_2, w_3, w_4, t)\\
&\to (w_0, w_1 - q_1(t) w_0, w_2 - q_2(t) w_0, 
w_3 - q_3(t) w_0, w_4 - q_4(t) w_0, t),
\end{split}
\end{equation}
we see that every quadric in
$H^0(\QQ) \otimes \BC[t]/(t^\mu)$ is singular at $(1,0,0,0,0)$ and thus
\eqref{E649} holds. Note that by the definition of $\mu$,
\begin{equation}\label{E915}
\left.\frac{d^{\mu} D_s}{d t^\mu}\right|_{t=0} \ne 0
\end{equation}
for $s\in H^0(\QQ)$ general.

Taking a general $s(w,t)\in H^0(\QQ)$, $s(w,0) = 0$ is a quadric of rank $4$
in $w_1, w_2, w_3, w_4$, where $w = (w_0, w_1, w_2, w_3, w_4)$. After an automorphism of $\PP^4$ preserving $q$, we can make $s(w,0)$ into the quadric
$w_1 w_4 - w_2 w_3$. Furthermore, there exists an automorphism $\lambda$ of $W/\Delta$ such that $\lambda$ preserves $q$ over the ring $\BC[t]/(t^\mu)$ and
\begin{equation}\label{E657}
s(\lambda(w,t)) = s_1(w,t) = w_1 w_4 - w_2 w_3 + t^\mu w_0^2.
\end{equation}
Due to the choice of $\lambda$, the partial derivatives of $s\in H^0(\QQ)$ in \eqref{E649} still vanish.

Applying an automorphism of the cone $\{s_1(w,0) = 0\}$, 
we can move $S$ to the cone in $\PP^4$ over the cubic rational normal curve
$\{(0, 1, x, x^2, x^3)\}$ with vertex at $(1,0,0,0,0)$. It is well known that $H^0(I_S(2))$ is spanned by
$w_1 w_4 - w_2 w_3$, $w_2^2 - w_1 w_3$ and $w_3^2 - w_2 w_4$
for the ideal sheaf $I_S$ of $S$ in $\PP^4$.
That is, $S$ is defined by
\begin{equation}\label{E655}
w_1 w_4 - w_2 w_3 = w_2^2 - w_1 w_3 = w_3^2 - w_2 w_4 = t = 0 
\end{equation}
in $W$. Correspondingly, we can complete $s_1$, given in \eqref{E657}, to a basis
$\{s_1,s_2, s_3\}$ of $H^0(\QQ)$,
as a free module over $\BC[[t]]$, satisfying
\begin{equation}\label{E654}
\begin{aligned}
s_1(w,0) &= w_1 w_4 - w_2 w_3\\
s_2(w,0) &= w_2^2 - w_1 w_3\\
s_3(w,0) &= w_3^2 - w_2 w_4.
\end{aligned}
\end{equation}
Furthermore, by a suitable choice of $s_i$, we may assume that
the expansions of $s_i$ do not contain the monomial terms
$t^k w_2w_3, t^k w_2^2$ and $t^k w_3^2$ for $k\ge 1$ because any of these terms
can be replaced by
\begin{equation}\label{E913}
\begin{aligned}
t^k w_2w_3 &= t^k w_1 w_4 - t^k s_1(w, t) + t^k(s_1(w,t) - s_1(w,0))\\
t^k w_2^2 &= t^k w_1 w_3 + t^k s_2(w,t) - t^k(s_2(w,t) - s_2(w,0))\\
t^k w_3^2 &= t^k w_2 w_4 + t^k s_3(w,t) - t^k(s_3(w,t) - s_3(w,0)).
\end{aligned}
\end{equation}
Then by replacing the basis $(s_1,s_2,s_3)$ by
\begin{equation}\label{E916}
\begin{bmatrix}
1\\
g(t) & A(t)
\end{bmatrix}
\begin{bmatrix}
s_1\\
s_2\\
s_3
\end{bmatrix}
\end{equation}
for some $g(t)\in \BC[[t]]$ and $A(t)\in \text{SL}(2,\BC[[t]])$, we can eliminate these terms.
In summary, we arrive at a basis $\{s_i\}$ of $H^0(\QQ)$ satisfying
\begin{equation}\label{E903}
\begin{aligned}
s_1(w,t) &= w_1 w_4 - w_2 w_3 + t^\mu w_0^2\\
s_2(w,0) &= w_2^2 - w_1 w_3\\
s_3(w,0) &= w_3^2 - w_2 w_4\\
\left.\frac{\partial^{j+1} s_i}{\partial w_0\partial t^j}\right|_{t=0} &\equiv 0
\\
\frac{\partial^3 s_i}{\partial t\partial w_2 \partial w_3} 
&= \frac{\partial^3 s_i}{\partial t\partial w_2^2} 
= \frac{\partial^3 s_i}{\partial t\partial w_3^2}
\equiv 0
\end{aligned}
\end{equation}
for $1\le i \le 3$ and $0\le j\le \mu-1$.

We have the syzygy relations
\begin{equation}\label{E666}
\begin{aligned}
w_2(w_1 w_4 - w_2 w_3) + w_3(w_2^2 - w_1 w_3) + w_1(w_3^2 - w_2 w_4) &= 0\\
w_3(w_1 w_4 - w_2 w_3) + w_4(w_2^2 - w_1 w_3) + w_2(w_3^2 - w_2 w_4) &= 0
\end{aligned}
\end{equation}
in $H^0(I_S(2))$. We claim that there exists $n\le \mu$ such that
\begin{equation}\label{E658}
\begin{aligned}
w_2 s_1 + w_3 s_2 + w_1 s_3 &= t^n f_1 + \sum_{j=1}^{n-1} t^j 
(a_{1j} w_1 s_1 + b_{1j} w_4 s_1 + c_{1j} w_1 s_3)\\
w_3 s_1 + w_4 s_2 + w_2 s_3 &= t^n f_2 + \sum_{j=1}^{n-1} t^j 
(a_{2j} w_1 s_1 + b_{2j} w_4 s_1 + c_{2j} w_4 s_2)
\end{aligned}
\end{equation}
for some $a_{ij}, b_{ij}, c_{ij}\in \BC$ and some $f_i(w,t)\in H^0(\CO_W(3))$ such that at least one of $f_1(w,0)$ and $f_2(w,0)$ does not vanish on $S$.
We achieve this by induction.

Suppose that we have \eqref{E658} for some $n\le \mu$ with no restriction on $f_i(w,t)$. This is obvious for $n=1$.
By comparing two sides of \eqref{E658}, we must have
\begin{equation}\label{E659}
\begin{aligned}
f_1(w,0) &= w_2 h_1(w) + w_3 h_2(w) + w_1 h_3(w)\\
f_2(w,0) &= w_3 h_1(w) + w_4 h_4(w) + w_2 h_5(w),
\end{aligned}
\end{equation}
for some $h_i\in H^0(\CO_P(2))$. Let us first handle the case $n < \mu$. 

When $n < \mu$, by comparing two sides of \eqref{E658} and using \eqref{E903},
we have
\begin{equation}\label{E911}
h_1(w)\equiv 0 \text{ and } 
\frac{\partial h_i}{\partial w_0} \equiv 0
\end{equation}
for all $i$. So
\begin{equation}\label{E912}
\begin{aligned}
f_1(w,0) &= w_3 h_2(w) + w_1 h_3(w)\\
f_2(w,0) &= w_4 h_4(w) + w_2 h_5(w).
\end{aligned}
\end{equation}
If one of $f_i(w,0)$ does not vanish on $S$, we are done. Otherwise, suppose that
$f_i(w,0)\in H^0(I_S(2))$ for $i=1,2$. Obviously, $f_1(w,0)\in H^0(I_S(2))$ if and only if
\begin{equation}\label{E905}
x^2 h_2(w_0, 1, x, x^2, x^3) \equiv - h_3(w_0, 1, x, x^2, x^3).
\end{equation}
Since $h_2$ and $h_3$ do not contain the monomial terms $w_2^2, w_2w_3$ and $w_3^2$ by our choice of $s_i$, we conclude that for every term
$\lambda w_a w_b$ in $h_2$, there is a term $-\lambda w_c w_d$ in $h_3$ such that
$a+b+2 = c+d$. Then
\begin{equation}\label{E914}
\begin{aligned}
h_2 &= \lambda_1 w_1^2 + \lambda_{2} w_1 w_2 + \lambda_{3} w_1 w_3 + \lambda_{4} w_1 w_4 + \lambda_5 w_2w_4\\
h_3 &= -\lambda_1 w_1 w_3 - \lambda_{2} w_1 w_4 - \lambda_{3} w_2 w_4 - \lambda_{4} w_3 w_4 - \lambda_5 w_4^2\\
f_1(w,0) &= -(\lambda_2 w_1 + \lambda_5 w_4) s_1(w,0) + \lambda_{3} w_1 s_3(w,0)
\end{aligned}
\end{equation}
and thus we can rewrite the first identity of \eqref{E658} as
\begin{equation}\label{E919}
w_2 s_1 + w_3 s_2 + w_1 s_3 = t^{n+1} \widehat{f}_1 + \sum_{j=1}^{n} t^j 
(a_{1j} w_1 s_1 + b_{1j} w_4 s_1 + c_{1j} w_1 s_3)
\end{equation}
for some $\widehat{f}_1(w,t)\in H^0(\CO_W(3))$. Applying the same argument to
$f_2$, we obtain
\begin{equation}\label{E661}
w_3 s_1 + w_4 s_2 + w_2 s_3 = t^{n+1} \widehat{f}_2 + \sum_{j=1}^{n} t^j 
(a_{2j} w_1 s_1 + b_{2j} w_4 s_1 + c_{2j} w_4 s_2)
\end{equation}
for some $\widehat{f}_2(w,t)\in H^0(\CO_W(3))$. We can continue this process until either one of $f_i(w,0)$ does not vanish on $S$ or $n=\mu$. Let us handle the case $n=\mu$.

When $m=\mu$, we have $h_1(w) = w_0^2$ and hence \eqref{E659} becomes
\begin{equation}\label{E910}
\begin{aligned}
f_1(w, 0) &= w_2 w_0^2 + w_3 h_2(w) + w_1 h_3(w)\\
f_2(w, 0) &= w_3 w_0^2 + w_4 h_4(w) + w_2 h_5(w).
\end{aligned}
\end{equation}
By examining the two sides of \eqref{E658} closer, we have the relation
\begin{equation}\label{E660}
\frac{\partial h_2}{\partial w_0} \equiv \frac{\partial h_4}{\partial w_0}
\text{ and }
\frac{\partial h_3}{\partial w_0} \equiv \frac{\partial h_5}{\partial w_0}
\end{equation}
among $h_i$. In particular, $h_2(q) = h_4(q)$ and $h_3(q) = h_5(q)$.
Therefore, the Jacobian
\begin{equation}\label{E663}
\begin{bmatrix}
\displaystyle{\frac{\partial f_1}{\partial w_1}} & 
\displaystyle{\frac{\partial f_1}{\partial w_2}} & 
\displaystyle{\frac{\partial f_1}{\partial w_3}} & 
\displaystyle{\frac{\partial f_1}{\partial w_4}}\\
\displaystyle{\frac{\partial f_2}{\partial w_1}} & 
\displaystyle{\frac{\partial f_2}{\partial w_2}} & 
\displaystyle{\frac{\partial f_2}{\partial w_3}} &
\displaystyle{\frac{\partial f_2}{\partial w_4}}
\end{bmatrix}_q
= \begin{bmatrix}
h_3(q) & w_0^2 & h_2(q) & 0\\
0 & h_5(q) & w_0^2  & h_4(q)
\end{bmatrix}
\end{equation}
of $f_i(w,0)$ has rank $2$ at $q$. Therefore, $\{f_1(w,0) = 0\}$ and $\{f_2(w,0) = 0\}$ are two cubic hypersurfaces in $\PP^4$ meeting transversely at $q$.
By the surjection
\begin{equation}\label{E656}
H^0(I_S(2)) \otimes H^0(\CO_P(1)) \twoheadrightarrow H^0(I_S(3)),
\end{equation}
every cubic hypersurface containing $S$ is singular at $q$. 
Therefore, neither $f_i(w,0)$ vanishes on $S$.
Indeed, $f_1(w,0)$ and $f_2(w,0)$ are linearly independent in $H^0(\CO_S(3))$
when $n = \mu$. This proves that there always exists $n\le \mu$ such that 
\eqref{E658} holds for some $f_i\in H^0(\CO_W(3))$ satisfying that at least one of $f_i(w,0)$ does not vanish on $S$.

Next we will show that $n=\delta$.
Let us consider the complete intersections
\begin{equation}\label{E906}
\begin{aligned}
Z_{i,a,b} &= \{s_{i+1} = s_{i+2}\\
&\hspace{18pt} = (aw_2 + bw_3) s_1 + (aw_3 + bw_4) s_2 + (aw_1 + bw_2) s_3 = 0\}\\
\end{aligned}
\end{equation}
in $W$ for $i=1,2,3$ and constants $a$ and $b\in \BC$, where we sets $s_{i+3} = s_i$. Obviously, $Z_{i,a,b}$ and $Z$ agree over an open set of $W$. More precisely,
\begin{equation}\label{E904}
\begin{aligned}
Z_{1,a,b} \cap \{a w_2 + bw_3 \ne 0 \} &= Z \cap \{a w_2 + bw_3 \ne 0 \}\\
Z_{2,a,b} \cap \{a w_3 + bw_4 \ne 0 \} &= Z \cap \{a w_3 + bw_4 \ne 0 \}\\
Z_{3,a,b} \cap \{a w_1 + bw_2 \ne 0 \} &= Z \cap \{a w_1 + bw_2 \ne 0 \}.
\end{aligned}
\end{equation}
For every point $p\ne q\in S\cap Y$, we can always find $i$ such that
$Z_{i,a,b} = Z$ locally at $p$ when $a$ and $b$ are general. Therefore, we have
\begin{equation}\label{E907}
\begin{aligned}
Z &= Z_{i,a,b} =
\{s_{i+1} = s_{i+2} = t^n (a f_1 + b f_2) = 0\}
\\
&=\{s_{i+1} = s_{i+2} = t^n = 0\} \cup \{s_{i+1} = s_{i+2} = a f_1 + b f_2 = 0\}
\end{aligned}
\end{equation}
locally at $p$. The scheme $\{s_{i+1} = s_{i+2} = t^n = 0\}$ is supported along $S$ with multiplicity $n$ and $\{s_{i+1} = s_{i+2} = a f_1 + b f_2 = 0\}$ does not contain $S$ since $a f_1(w,0) + b f_2(w,0)\not\in H^0(I_S(3))$ for general choices of $a$ and $b$. In conclusion, we must have $n = \delta$ and 
\begin{equation}\label{E662}
Y = \{ s_{i+1} = s_{i+2} = a f_1 + b f_2 = 0 \}
\end{equation}
locally at $p$ for $a$ and $b$ general.
It follows that $f_j(w,t)\in H^0(I_Y(3))$ for $j=1,2$.
This also gives a more explicit proof that $Y$ is a l.c.i outside of $q$.

So far we have proved that $n=\delta \le \mu$. 
Next, we claim that $\delta = \mu$.
By \eqref{E662}, we have
\begin{equation}\label{E909}
Y_0\subset S\cap \{f_1(w,0) = f_2(w,0) = 0\}.
\end{equation}
When $n < \mu$, this implies
\begin{equation}\label{E920}
Y_0\subset S\cap \{w_3 h_2 + w_1 h_3 = w_4 h_4 + w_2 h_5 = 0\}.
\end{equation}
by \eqref{E912}. All quadrics $h_i(w) = 0$ are singular at $q$
for $2\le i\le 5$ by \eqref{E911}. 
It follows that the right hand side of \eqref{E920}
is a union of lines on $S$ through $q$, which contradicts our hypotheses on $Y$. Therefore, $n = \mu$.

In conclusion, $n = \mu = \delta$ and $f_i(w,0)$ define two cubic hypersurfaces in $\PP^4$ meeting transversely at $q$.

It remains to study the local behavior of $Y$ at the vertex $q$.
Let us consider the complete intersection
\begin{equation}\label{E664}
\begin{aligned}
Y' &= \{ s_1(w,t) = f_1(w,t) = f_2(w,t) = 0\}
\\
&= \{
w_1 w_4 - w_2 w_3 + t^\delta w_0^2 = f_1(w,t) = f_2(w,t) = 0
\}
\end{aligned}
\end{equation}
in $W$. Obviously, $Y\subset Y'$.

By \eqref{E663}, using implicit function theorem, we see that $Y'$
is locally given by $\{y^2 = g(x,t)\}\subset \Delta_{xyt}^3$ at $q$
for some $g(x,t)\in \BC[[x,t]]$.
The fact that $g(x,t)$ satisfies \eqref{E644} follows from the observation that
\begin{equation}\label{E921}
\frac{\partial^3 f_i}{\partial w_0^3} \equiv 0
\end{equation}
for $i=1,2$, i.e., the expansions of $f_i(w,t)$ do not contain the terms $t^k w_0^3$. One can see this from \eqref{E658}.

It follows that
$Y'$ is a reduced l.c.i at $q$ and is either irreducible or
the union of two irreducible components $y = \pm \sqrt{g(x,t)}$ at $q$.
If it is the latter and $Y\subset Y'$ is one of the two components, then
$Y_0$ is smooth at $q$; on the other hand, $Y_0$ is a curve of degree $8$
on the cubic cone $S$ and thus must be singular at $q$. Therefore,
$Y = Y'$ in an open neighborhood of $q$. So $Y$ is a l.c.i everywhere.

If $g(x,0)\equiv 0$, then $Y_0$ is locally isomorphic to
$\{y^2 = 0\}\subset \Delta_{xy}^2$ and hence nonreduced at $q$.

It is not hard to figure out the curve $Y_0\subset S$.
Note that each cubic hypersurface $\{f_i(w,0) = 0\}$ cuts out a curve of degree $9$ on $S$, while $Y_0$ is a curve of degree $8$. Therefore,
\begin{equation}\label{E933}
S\cap \{f_i(w,0) = 0\} = Y_0 + L_i
\end{equation}
for $i=1,2$ and two lines $L_1$ and $L_2$ on $S$ through $q$. Since $Y = Y'$ locally at $q$, $L_1$ and $L_2$ must be distinct.
For a general line $L$ on $S$ through $q$, $L$ meets each cubic hypersurface $\{f_i(w,0) = 0\}$ transversely at $q$ since $\{f_i(w,0) = 0\}$ are smooth at $q$; consequently, $L$ meets $Y_0$ at two points transversely or one point with multiplicity $2$
outside of $q$. This tells us the class of the curve $Y_0$.
Let $\nu: \widehat{S} \cong \BF_3\to S$ be the resolution of singularity of $S$ and $\Gamma$ be the proper transform of $Y_0$ under $\nu$. Then $\nu^* \CO_S(1)$ has degree $8$ on $\Gamma$ and $\Gamma$ meets a fiber of $\BF_3 \to \PP^1$ with intersection number $2$. That is,
\begin{equation}\label{E934}
\Gamma . C = 8 \text{ and } \Gamma . F = 2
\end{equation}
where $C$ and $F$ are effective generators of $\Pic (\BF_3)$ with
$C^2 = 3$, $C.F=1$ and $F^2 = 0$. Obviously, \eqref{E934} determines the class of $\Gamma$, which is $\Gamma\in |2C + 2F|$.

When $Y_0$ is integral, the surjection \eqref{E932} follows from the diagram
\begin{equation}\label{E935}
\begin{tikzcd}
\Sym^2 H^0(\CO_{Y_0}(1)) \ar[two heads]{r} \ar[equal]{d}
& H^0(\CO_{Y_0}(2)) \ar[equal]{d}\\
\Sym^2 H^0(\CO_\Gamma(C)) \ar[two heads]{r} \ar[equal]{d}
& H^0(\CO_{\Gamma}(2C)) \ar[equal]{d}
\\
\Sym^2 H^0(\BF_3, C) \ar[two heads]{r} \ar[equal]{d}
& H^0(\BF_3, 2C) \ar[equal]{d}
\\
\BC^{15} \ar[two heads]{r} & \BC^{12}
\end{tikzcd}
\end{equation}
\end{proof}

Most of Theorem \ref{THMDELTAP} follows directly from Lemma \ref{LEMDELTAP}. Only the second last statement needs an additional argument. That is, we need to justify that 
$n\le 6$ in \eqref{E642} if $g(Y_p) = 0$. It follows from the lemma below.

\begin{lem}\label{LEMRCCC}
Let $C$ be a rational curve of $\deg C = 8$ and $p_a(C) = 5$ lying on a cubic cone in $\PP^4$. If 
there exists a hyperplane $\Lambda$ of $\PP^4$ meeting $C$ at a unique smooth point and $C$ has a double singularity $\{y^2 = x^n\}\subset \Delta_{xy}^2$ at the vertex $q$ of the cone, then $n\le 6$.
\end{lem}

\begin{proof}
Let $\nu: \widehat{C} \cong \PP^1\to \PP^4$ be the normalization of $C$. We can make everything very explicit. Suppose that the cubic cone is given by 
\begin{equation}\label{E940}
w_1 w_4 - w_2 w_3 = w_2^2 - w_1 w_3 = w_3^2 - w_2 w_4
\end{equation}
as usual with $q = (1,0,0,0,0)$ and $(w_i)$ the homogeneous coordinates of $\PP^4$.

Let $t$ be the affine coordinate of $\widehat{C}$. We identify $H^0(\CO_{\widehat{C}}(d))$ with the space of polynomials in $t$ of degree $\le d$.

Obviously, $\Lambda$ does not pass through $q$. Using $\Aut(\PP^4)$, we may assume that $\Lambda = \{w_0=0\}$. By our hypotheses, $\nu^{-1}(\Lambda)$ consists of a single point, which we take to 
be $\infty$. So $\nu^*(\Lambda) = \nu^* (w_0) = 1 \in H^0(\CO_{\widehat{C}}(8))$.

If $n$ is odd, $\nu^{-1}(q)$ consists of a single point, say $0$; if $n$ is even, $\nu^{-1}(q)$ consists of two points, say $0$ and $1$. We let $s = t^2$ if $n$ is odd and $s = t(t-1)$ if $n$ is even.
Since a general hyperplane passing through $q$ meets $C$ at $q$ with multiplicity $2$, $\nu$ is given by
\begin{equation}\label{E942}
\nu(t) = (1, s f_1, s f_2, s f_3, s f_4)
\end{equation}
for some $\{f_i\}$ spanning a base point free linear system in $H^0(\CO_{\widehat{C}}(6))$. And since $C$ lies on the cubic cone 
\eqref{E940}, it is easy to see that
\begin{equation}\label{E924}
\nu(t) = (1, sf^3, sf^2 g, sf g^2, sfg^3)
\end{equation}
for some $\{f,g\}$ spanning a base point free linear system in $H^0(\CO_{\widehat{C}}(2))$. Note that we are free to replace $\{f,g\}$ by a basis of $\Span \{f,g\}$ using $\PP \text{GL}(5)$ actions that preserves the cubic cone. Let us choose $f(t)$ and $g(t)$ such that $\gcd(f,s) = 1$ and $g(0) = 0$.

We base our argument for $n\le 6$ on the following observation: there does not exist a hypersurface $Q$ in $\PP^4$ such that the intersection multiplicity $(Q.C)_q$ at $q$ is an odd integer less than $n$; also if $n$ is even and $Q$ meets one branch of $C$ at $q$ with multiplicity $m < n/2$, then it must meet the other branch with the same multiplicity $m$.
Here we will find a quadric $Q$ such that $(Q.C)_q = 5$ if $n$ is odd and $(Q.C)_q = 7$ if $n$ is even. This will imply $n\le 6$.

Let us consider the hyperplane $\Lambda' = \{w_2 = 0\}$. By our choice of $f$ and $g$, $3\le (\Lambda' . C)_q \le 4$. If $(\Lambda' . C)_q = 3$, we are done; otherwise, $(\Lambda' . C)_q = 4$. 
If $n$ is odd, this implies that $g$ is a multiple of $s$. If $n$ is even, $g = cs$ or $g = ct^2$ for some $c\ne 0$. If it is the latter, $\Lambda'$ meets the two branches of $C$ at $q$ with multiplicities $2$ and $4$, respectively, and hence $n\le 4$. So we conclude that $g$ is a multiple of $s$; otherwise, we are done. Let us assume that $g = s$. So for every $f$ such that $\{f,g\}$ is a basis of $\Span\{f,g\}$, we have $\gcd(f,s) = 1$. Thus,
we may choose $f$ such that $f(\infty) = 0$. That is, we may simply take $f = t - a$ for some $a$ such that $\gcd(f, s) = 1$. More explicitly, $\nu$ is
\begin{equation}\label{E949}
\nu(t) = (1, f^3 s, f^2 s^2, f s^3, s^4),
\end{equation}
for $f = t-a$, where $s=t^2$ and $a\ne 0$ if $n$ is odd and
$s=t(t-1)$ and $a\ne 0,1$ if $n$ is even.

Let $Q$ be the quadric defined by $w_1^2 - a^4 w_0 w_2 = 0$. Then
\begin{equation}\label{E950}
\nu^*(Q) = (f^3 s)^2 - a^4 f^2 s^2 = f^2 s^2(f^4 - a^4).
\end{equation}

When $n$ is odd, $f^4 - a^4 = (t-a)^4 - a^4$ vanishes at $0$ of order $1$. Hence $(Q.C)_q = 5$ and this proves $n\le 6$ for $n$ odd.

When $n$ is even, $\nu^* Q$ vanishes at $0$ of order $3$ and at $1$ of order $2$ or $3$. If $\nu^* Q$ vanishes at $1$ of order $2$, $(Q.C)_q = 5$ and we are done. Otherwise, suppose that $\nu^* Q$ vanishes at $1$ of order $3$, which happens if and only if
$(f(1))^4 - a^4 = (1-a)^4 - a^4 = 0$.

Now we choose another quadric $Q' = \{w_1^2 - a^4 w_0 w_2 + 4a^4 w_0 w_3 = 0\}$. Then
\begin{equation}\label{E943}
\nu^* (Q') = fs^3 \left(\frac{(f^4 - a^4)f}{s} + 4a^4\right)
\end{equation}
where an easy computation shows
\begin{equation}\label{E941}
\frac{(f^4 - a^4)f}{s} = \frac{((t-a)^4 - a^4)(t-a)}{t(t-1)}
= \begin{cases}
- 4a^4 & \text{if } t=0\\
5(1-a)^4 -a^4 & \text{if } t=1.
\end{cases}
\end{equation}
Therefore, $\nu^* (Q')$ vanishes at $0$ of order at least $4$. 
On the other hand, since $a^4 = (1-a)^4\ne 0$, $\nu^*(Q')$ vanishes at $1$ of order $3$. Consequently, $n\le 6$.
\end{proof}

\section{Numerical Computations on the Pseudo Relative Canonical Model}\label{SECNC}

\subsection{Numerical Invariants of $W$ and $Z$}

We have constructed a pseudo relative canonical model $Y/C$ of $X/C$
in the previous section. The local study of $Y$ has made it possible to 
carry out numerical computations on $Y$, as a subvariety of $W$.
We start with the computation of a few simple numerical invariants of $W$ and $Z$.

Note that $\LL = \CO_X(8\Gamma) = K_X\otimes \CO_X(-V - f^*D)$,
\begin{equation}\label{E567}
h^0(\LL) = 1\text{ and } H^0(\LL\otimes \CO_X(f^*D)) = H^0(K_X)
= H^0(\CO_C(D)))
\end{equation}
by \eqref{E529}.
Thus, combining with \eqref{E522}, we see that
\begin{equation}\label{E568}
f_*\LL = \CO_C \oplus A
\end{equation}
where $A$ is a vector bundle of rank $4$ on $C$ satisfying that
\begin{equation}\label{E593}
H^0(A(D)) = H^0(A\otimes \CO_C(D)) = 0.
\end{equation}
The fact that $H^0(A(D)) = 0$ implies that 
\begin{equation}\label{E581}
c_1(A) \le -4d - 4 \chi(\CO_C)
\end{equation}
by Riemann-Roch. 

Clearly, $\CH^\bullet(W)$ is generated by $\pi^* \CH^1(C)$ and $\Lambda$
with the relation
\begin{equation}\label{E554}
\Lambda^5 = \Lambda^4 . \pi^* c_1(f_*\LL) = c_1(A),
\end{equation}
where $\pi$ is the projection $W\to C$. By the Euler sequence
\begin{equation}\label{E917}
\begin{tikzcd}
0 \ar{r} & \CO_W \ar{r} & \pi^*(f_*\LL)^\vee \otimes \CO_W(\Lambda) \ar{r}
& T_{W/C} \ar{r} & 0
\end{tikzcd}
\end{equation}
we obtain the Chern character
\begin{equation}\label{E918}
\begin{aligned}
\ch(T_W) &= \pi^*\ch(T_C) + \exp(\Lambda) (1 + \pi^* \ch(A^\vee))
- 1\\
&= 5 - \pi^* K_C + 5\Lambda - \pi^* c_1(A) \\
&\quad + \sum_{n\ge 2} \left(
\frac{5}{n!} \Lambda^n - \frac{1}{(n-1)!} \Lambda^{n-1} \pi^* c_1(A)
\right)
\end{aligned}
\end{equation}
of the tangent bundle $T_W$ and it follows that
\begin{equation}\label{E602}
\omega_W = -5\Lambda + \pi^* K_C + \pi^* c_1(A).
\end{equation}

The ideal sheaf $I_Z$ of $Z\subset W$ can be resolved by the Koszul complex 
\begin{equation}\label{E584}
0\xrightarrow{} \wedge^3 \pi^* \QQ \xrightarrow{} \wedge^2 \pi^* \QQ
\xrightarrow{} \pi^* \QQ \xrightarrow{} I_Z\otimes \CO_W(2) \xrightarrow{} 0.
\end{equation}
It follows that
\begin{equation}\label{E578}
Z = -c_3(\pi^* \QQ \otimes \CO_W(-2)) = 8\Lambda^3 - 4\Lambda^2 \pi^* c_1(\QQ)
\end{equation}
in $\CH_2(W)$ and
\begin{equation}\label{E576}
\N_Z = (I_Z/I_Z^2)^\vee = (\pi^* \QQ)^\vee \otimes \CO_W(2) \Big|_Z
\end{equation}
where $\N_Z$ is the normal sheaf of $Z$ in $W$.
In general, we use the notation $\N_{F/G}^\vee = I_F/I_F^2$ to denote the conormal sheaf of $F\subset G$ and $\N_{F/G}$ for its dual, namely, the normal sheaf, where $I_F$ is the ideal sheaf of $F$ in $G$. If the context is clear about the ambient space $G$, we simply write $\N_F = \N_{F/G}$.

\subsection{Intersection Number $\Lambda^2 Y$}

Since
\begin{equation}\label{E569}
\Lambda\Big|_Y = 8\Gamma_Y\text{ and } \tau^*\Lambda = 8\Gamma,
\end{equation}
we have
\begin{equation}\label{E571}
\Gamma^2 = \Gamma_Y^2 = \frac{1}8 \Lambda \Gamma_Y = \frac{1}{64} \Lambda^2 Y,
\end{equation}
where the intersection number $\Gamma_Y^2$ is taken in $Y$.

We need to figure out $\Lambda^2 Y$. By \eqref{E572},
\begin{equation}\label{E573}
\Lambda^2 Y = \Lambda^2 Z - \sum_{p\in C}\delta_p \Lambda^2 S_p
= \Lambda^2 Z - 3 \delta
\end{equation}
since $\deg S_p = 3$ in $W_p\isom \PP^4$,
where $\delta = \sum \delta_p$.

Combining \eqref{E554}, \eqref{E578}, \eqref{E571} and \eqref{E573}, 
we obtain
\begin{equation}\label{E579}
\begin{aligned}
64 \Gamma^2 &= 64\Gamma_Y^2 = 8\Lambda \Gamma_Y = \Lambda^2 Y=
8\Lambda^5 - 4\Lambda^4 \pi^* c_1(\QQ) - 3\delta\\
&= 8 c_1(A) - 4 c_1(\QQ) - 3\delta. 
\end{aligned}
\end{equation}

\subsection{Dualizing Bundles $\omega_Y$ and $\omega_Z$}

By \eqref{E602}, \eqref{E576} and adjunction, we have
\begin{equation}\label{E623}
\omega_Z = \omega_W \Big|_Z + c_1(\N_Z) =
(\Lambda + \pi^* K_C + \pi^* c_1(A) - \pi^* c_1(\QQ))\Big|_Z.
\end{equation}
At every smooth point of $Y_p$, $Z$ is locally given by $y t^{\delta_p} = 0$ in
$\Delta_{xyt}^3$ with $Y = \{ y = 0\}$ and $S_p = \{t = 0\}$, by the local analysis carried out in the proof of Lemma \ref{LEMDELTAP}. Therefore,
\begin{equation}\label{E624}
\begin{aligned}
\omega_Y &= \omega_Z\Big|_Y - \sum_{p\in C} \delta_p Y_p\\
&= \pi_Y^*(\underbrace{ K_C + c_1(A) - c_1(\QQ) - \sum \delta_p p}_{D_Y}
) + 8\Gamma_Y
\end{aligned}
\end{equation}
where $\pi_Y$ is the projection $Y\to C$.

\subsection{Adjunction on $(Y, \Gamma_Y)$}

By Theorems \ref{THM3QUADRICS} and \ref{THMNORMAL}, $Y$ is smooth along $\Gamma_Y$.
Therefore, adjunction applies to $(Y,\Gamma_Y)$ and produces
\begin{equation}\label{E625}
\begin{aligned}
c_1(K_C) &= c_1(K_{\Gamma_Y}) = (\omega_Y + \Gamma_Y) \Gamma_Y\\
&= c_1(K_C) + c_1(A)  - c_1(\QQ) - \delta + 9\Gamma_Y^2
\\
&= c_1(K_C) + \frac{17}{8} c_1(A) - \frac{25}{16} c_1(\QQ) - \frac{91}{64}\delta
\\
&\Rightarrow
c_1(\QQ) = \frac{34}{25} c_1(A) - \frac{91}{100}\delta.
\end{aligned}
\end{equation}
Combining \eqref{E579} and \eqref{E625}, we reach the following crucial identity
\begin{equation}\label{E580}
\boxed{\delta = 100\Gamma^2 - 4c_1(A)}.
\end{equation}

So far by computing $\Lambda\Gamma_Y$ in two different ways, we have established the numerical relations among $c_1(A), c_1(\QQ), \delta$ and $\Gamma^2$ in \eqref{E579} and \eqref{E580}. Next we are trying to bound $\delta$ and $c_1(A)$ in terms of the double surface singularities $Y$ has at the vertices of the cones $S_p$. For that purpose, let us first introduce the invariant $b$.

\subsection{Local Invariants $b_{p,q}$ and $h^1(\CO_\E)$}

We have
\begin{equation}\label{E628}
\begin{aligned}
f^* D + 8 \Gamma + V = K_X &= \tau_Y^* \omega_Y + \sum a(E, Y) E
\\
&= f^* D_Y + 8\Gamma + \sum a(E,Y) E 
\end{aligned}
\end{equation}
where $D_Y$ is the divisor on $C$ defined in \eqref{E624}, $E$ runs over
all exceptional divisors of $\tau$ and $a(E,Y)$ is the discrepancy of $E$
with respect to $Y$. 
For each fiber $Y_p$ and every point $q\in Y_p$, we let $b_{p,q}$ be the smallest
non-negative rational number such that the divisor
\begin{equation}\label{E629}
\sum_{\tau(E) = q} a(E,Y) E + b_{p,q} X_p
\end{equation}
is $\BQ$-effective. Clearly, $b_{p,q} = 0$ when $Y$ has at worst canonical
singularities at $q$. More explicitly, $b_{p,q}$ is simply given by
\begin{equation}\label{E630}
b_{p,q} = \max_{\tau(E) = q} \left(0, -\frac{a(E,Y)}{\mult_E(X_p)}\right)
\end{equation}
where $\mult_E(X_p)$ is the multiplicity of the component $E$ in $X_p$. One can define $b_{p,q}$ for any fiberation $Y/C$. It is independent of the choice of the desingularization of $Y$. For a smooth surface $X$ relative minimal over $C$, $a(E,Y) \le 0$ for all $E\subset X_p$ and so we actually have
\begin{equation}\label{E930}
b_{p,q} = \max_{\tau(E) = q} \left(-\frac{a(E,Y)}{\mult_E(X_p)}\right).
\end{equation}
If we let
\begin{equation}\label{E928}
b_p = \max_{q\in Y_p} b_{p,q},
\end{equation}
we can put \eqref{E628} in the form
\begin{equation}\label{E631}
f^* D  + V = f^* \left(D_Y - \sum b_p p\right) + 
\left(\sum a(E,Y) E + \sum b_p X_p\right).
\end{equation}
Obviously, either side of \eqref{E631} is a Zariski decomposition of the same
divisor. By the uniqueness of Zariski decomposition, we conclude that
\begin{equation}\label{E632}
V = \sum a(E,Y) E + \sum b_p X_p.
\end{equation}
And since $\Gamma E = 0$ for all exceptional divisors $E$ of $\tau$, we have
\begin{equation}\label{E633}
\Gamma V = \sum_{p\in C} b_p.
\end{equation}
Note that \eqref{E632} implicitly says that $b_p\in \BN$ for all $p$.
We will only consider the double surface singularities of $Y$ at the vertices of
the cones $S_p$. Let
\begin{equation}\label{E926}
V' = \sum_{\delta_p\ne 0} \left(\sum_{\tau(E) = q} a(E,Y) E + 
\lceil b_{p,q} \rceil X_p\right),
\end{equation}
where $q\in Y_p$ is the vertex of $S_p$.
Clearly, $V' \le V$ and we let
\begin{equation}\label{E929}
b = \sum_{\delta_p\ne 0} \lceil b_{p,q} \rceil
\text{ and } \E = -\sum_{\delta_p\ne 0} \sum_{\tau(E) = q} a(E,Y) E.
\end{equation}
We may give a better bound for $c_1(A)$ than \eqref{E581} in terms of $b$ and $\E$.

Note that
\begin{equation}\label{E947}
c_1(A) = c_1(f_*\LL) = h^0(8\Gamma + f^* N) - 5\deg N - 5\chi(\CO_C)
\end{equation}
for a sufficiently ample divisor $N$ on $C$.
The dimension $h^0(8\Gamma + f^* N)$ can be estimated by
\begin{equation}\label{E936}
\begin{aligned}
&\quad 8 \Gamma + f^* N = K_X + f^*(N-D) - V\\
&\le K_X + f^*(N-D) - V'\\
&= K_X + f^*(N-D) - \sum_{\delta_p\ne 0} \lceil b_{p,q} \rceil X_p
+ \E.
\end{aligned}
\end{equation}
For $M = K_X + f^*(N-D) - \sum \lceil b_{p,q} \rceil X_p$, by the standard exact sequence
\begin{equation}\label{E945}
\begin{tikzcd}
0 \ar{r} & \CO(M)\ar{r} & \CO(M+\E) \ar{r} & \CO_\E(M+\E) \ar{r} & 0
\end{tikzcd}
\end{equation}
we obtain
\begin{equation}\label{E937}
\begin{aligned}
&\quad h^0(8 \Gamma + f^* N) \le h^0(M+\E) \le h^0(M) + h^0(\CO_\E(M+\E))
\\
&= h^0(K_X + f^*(N-D) - \sum \lceil b_{p,q} \rceil X_p) + h^0(\omega_\E)\\
&= h^0(K_X + f^*(N-D) - \sum \lceil b_{p,q} \rceil X_p) + h^1(\CO_\E)
\end{aligned}
\end{equation}
where $\omega_\E$ is the dualizing sheaf of $\E$. Note that
\begin{equation}\label{E922}
\omega_\E = \CO_\E(M + \E) = \CO_\E(K_X + \E) = \CO_\E(\tau_Y^* \omega_Y) = \CO_\E
\end{equation}
and hence $h^0(\CO_\E) = h^1(\CO_\E)$.
It is also worthwhile to point out that $\E$ is the closed subscheme of $X$ defined by
\begin{equation}\label{E938}
\CO_\E = \frac{\CO_X}{\CO_X(-\E)}.
\end{equation}
If $\E = 0$, we take $\CO_\E = 0$ and $h^1(\CO_\E) = 0$.

On the other hand, $h^0(M)$ can be easily estimated by
\begin{equation}\label{E944}
\begin{aligned}
&\quad h^0(K_X + f^*(N-D) - \sum \lceil b_{p,q} \rceil X_p)
\\
&\le h^0(K_X) + 5 (\deg N - \deg D) - 5 \sum \lceil b_{p,q} \rceil
\\
&= \chi(\CO_C) + 5\deg N - 4d - 5 b
\end{aligned}
\end{equation}
which follows from the exact sequence
\begin{equation}\label{E923}
\begin{tikzcd}[column sep = small]
0 \ar{r} & \CO_X(K_X) \ar{r} & \CO_X(K_X + f^* G) \ar{r} &
\displaystyle{\sum_{p\in G}\CO_{X_p}(K_{X_p})} \ar{r} & 0.
\end{tikzcd}
\end{equation}
Combining \eqref{E947}, \eqref{E937} and \eqref{E944}, we obtain
\begin{equation}\label{E980}
c_1(A) \le -4(d+\chi(\CO_C)) - 5b + h^1(\CO_\E).
\end{equation}
The extra term $-(5b - h^1(\CO_\E))$ makes it a better bound than \eqref{E581}.

\subsection{Key Inequality}

We claim that
\begin{equation}\label{E605}
\boxed{
\begin{aligned}
\chi(X) - \chi(F) \chi(C) &= 12 \chi(\CO_X) - K_X^2 + 16 \chi(\CO_C) \\
&\ge \delta - 28b + 4 h^1(\CO_\E) + K_X \E
\end{aligned}
}
\end{equation}
where $F$ is a general fiber of $f$ and $\chi(X), \chi(F)$ and $\chi(C)$ are
the (topological) Euler characteristics of $X, F$ and $C$, respectively. The
left hand side of \eqref{E605} measures, roughly, how far $X$ is from a
smooth fibration; clearly, $\chi(X) = \chi(F)\chi(C)$ if every fiber of $f$
is smooth.

Let us first see how to derive our main theorem from \eqref{E580} and
\eqref{E605}. We observe that
\begin{equation}\label{E505}
\chi(\CO_X) = h^0(K_X) - h^1(\CO_X) + 1 \le h^0(K_X) - h^1(\CO_C) + 1
= d + 2 \chi(\CO_C).
\end{equation}
Combining \eqref{E505} and \eqref{E529}, we obtain
\begin{equation}\label{E606}
\begin{aligned}
&\quad 12 \chi(\CO_X) - K_X^2 + 16 \chi(\CO_C)
\\
&\le 12(d + 2\chi(\CO_C)) - K_X(f^* D + 8\Gamma + V) + 16 \chi(\CO_C)\\
&\le 12(d + 2\chi(\CO_C)) - K_X(f^* D + 8\Gamma + V') + 16 \chi(\CO_C)\\
&= 4d + 8 \Gamma^2 + 56 \chi(\CO_C) - 8b + K_X\E,
\end{aligned}
\end{equation}
where $V'\le V$ is given by \eqref{E926}. Therefore, we conclude
\begin{equation}\label{E607}
4d + 8 \Gamma^2 + 56 \chi(\CO_C) \ge \delta - 20 b + 4h^1(\CO_\E)
\end{equation}
by \eqref{E605}. Thus, replacing $\delta$ in \eqref{E607}
by the right hand side of \eqref{E580}, we obtain
an upper bound for $\Gamma^2$:
\begin{equation}\label{E583}
\Gamma^2 \le \frac{d + c_1(A) + 5b - h^1(\CO_\E)}{23} + \frac{14}{23} \chi(\CO_C).
\end{equation}
Consequently,
\begin{equation}\label{E612}
\Gamma^2 \le -\frac{3}{23} d + \frac{10}{23} \chi(\CO_C)
\end{equation}
by \eqref{E980}.
Combining \eqref{E612} with \eqref{E530}, we have
\begin{equation}\label{E626}
-\frac{d}8 - \frac{17}4 \chi(\CO_C)
\le -\frac{3}{23} d + \frac{10}{23} \chi(\CO_C),
\end{equation}
which fails if $p_g(X) = d+ \chi(\CO_C) > 863 \chi(\CO_C)$.
This proves Theorem \ref{THMGENUS5}, provided that \eqref{E605} holds.
The rest of the paper will be devoted to the proof of \eqref{E605}.
It is organized as follows:
\begin{itemize}
\item
First, we will interpret the left hand side of \eqref{E605} as a moduli invariant of the family $X/C$ and show that it can be computed via local invariants of the singular fibers of $X/C$. 
Much of this was due to S.L. Tan and we are not claiming any originality on our part.
In this way, we will turn \eqref{E605} into a local problem on the double surface singularities of $Y$ at the vertices of the cones $S_p$.
\item
Second, we will try to give an explicit upper bound for the right hand side of \eqref{E605}. Especially, we need to understand $h^1(\CO_\E)$, which is quite mysterious. For one thing, it is not even clear that $h^1(\CO_\E) \le 5b$ in \eqref{E980}. We will prove this and more.
\item
Finally, we will reduce \eqref{E605} to a local inequality on the double surface singularties of $Y$ and prove it in \S \ref{SECDSS}.
\end{itemize}

\subsection{Moduli Invariants}

If $f: X\to C$ is semistable, the left hand side of \eqref{E605} is given by the moduli invariants $c_1(f_*\omega_f)$ and $f_*(\omega_f^2)$
in (cf. \cite[p. 154]{H-M})
\begin{equation}\label{E582}
\chi(X) - \chi(F) \chi(C) = 12 \deg f_*\omega_f - \omega_f^2 
= \eta
\end{equation}
where $\eta$ is the total number of nodes of the singular
fibers of $f$. This follows from Grothendieck-Riemann-Roch and
the exact sequence
\begin{equation}\label{E608}
\begin{tikzcd}
0 \arrow{r} & f^*\Omega_C \arrow{r} & \Omega_X \arrow{r} & \omega_f
\arrow{r} & \omega_f\otimes \CO_\Sigma \arrow{r} & 0
\end{tikzcd}
\end{equation}
where $\Sigma\subset X$ is the subscheme of the nodes of the singular fibers of $f$.

We want to generalize \eqref{E582} to fibrations $f: X\to C$ of curves
with the only hypotheses that both $X$ and $C$ are smooth and projective. Of course,
we can still define $\eta$ to be the LHS of \eqref{E582}. The point here is to
understand how the singular loci of the fibers of $f$ contribute to $\eta$.

As mentioned at the very beginning, this problem had been extensively
studied by S.L. Tan. For our purpose, we just have to deal with the
contribution coming from the resolution of a double surface singularity
of type \eqref{E642}.
So we are going to give a simple treatment to this problem.
Please see \cite{T1} and \cite{T2} for a comprehensive solution.

Basically, we need to figure out
how to amend the exact sequence \eqref{E608} when $X/C$ is not necessarily
semistable. The trouble here is, of course, that $X_p$ might be nonreduced
for some $p$. In case that $X/C$ has reduced fibers, \eqref{E608} continues
to hold with $\Sigma$ being the subscheme supported at the singularities of $X_p$ and
locally defined by the Jacobian ideal at each singular point. More precisely,
if $X$ is locally given by $g(x,y)=t$ at $q\in X$, $\Sigma$ is locally cut out
by $\partial g/\partial x = \partial g/\partial y = 0$, where $t$ is the local
parameter of the base $C$. Thus, a node
of $X_p$ at $q$ (i.e. $g(x,y)=xy$) contributes $1$ to $\eta$,
a cusp of $X_p$ at $q$ (i.e. $g(x,y) = x^2 - y^3$) contributes $2$ to $\eta$ and so on.

Let us first see how to define the map $\Omega_X \to \omega_f$ in general.
For a vector bundle (or just a coherent sheaf) $E$ over $X$,
we always has a map
\begin{equation}\label{E609}
\begin{tikzcd}
E \arrow{r}{\wedge s} & \wedge^2 E
\end{tikzcd}
\end{equation}
after fixing a section $s\in H^0(E)$. The natural map $f^* \Omega_C \to \Omega_X$
gives rise to
a section of $\Omega_X\otimes (f^*\Omega_C)^\vee$ and thus induces a map
\begin{equation}\label{E610}
\begin{tikzcd}
\Omega_X\otimes (f^*\Omega_C)^\vee \arrow{r}\arrow[equal]{d} &
\wedge^2 (\Omega_X\otimes (f^*\Omega_C)^\vee)\arrow[equal]{d}\\
\Omega_X \otimes f^* K_C^{-1} \arrow{r} & \omega_f \otimes f^* K_C^{-1}.
\end{tikzcd}
\end{equation}
Therefore, we have the exact sequence
\begin{equation}\label{E611}
\begin{tikzcd}
0\arrow{r} & \ker(\rho) \arrow[equal]{d}\arrow{r} &
\Omega_X \arrow{r}{\rho} & \omega_f \arrow{r} & \coker(\rho)\arrow[equal]{d}\arrow{r} & 0\\
& L &&& \omega_f\otimes \CO_\Sigma
\end{tikzcd}
\end{equation}
as a generalization of \eqref{E608}, where
an easy local study shows that $L = \ker(\rho)$ is a line bundle on $X$ and
$\Sigma$ is the subscheme of $X$ defined by the Jacobian ideals of $X_p$ as above. More explicitly, if we let
$\Sigma = \Sigma_0 \cup \Sigma_1$ with $\Sigma_0$ of pure dimension $0$ and
$\Sigma_1$ of pure dimension $1$, then
\begin{equation}\label{E613}
\Sigma_1 = \sum (\mult_G(X_p) - 1) G = \sum_{p\in C} \big(X_p - (X_p)_\text{red}\big)
\end{equation}
where $G$ runs over all irreducible components of $X_p$ for all $p$
and $(X_p)_\text{red}$ is the largest reduced subscheme supported on $X_p$.

Taking Chern characters of \eqref{E611}, we obtain
\begin{equation}\label{E614}
\begin{split}
&\quad \ch(\omega_f) - \ch(\Omega_X) = \ch(\CO_{\Sigma_0}) + 
\ch(\omega_f \otimes\CO_{\Sigma_1})
-\ch(L)\\
&=\ch(\CO_{\Sigma_0}) + (\exp(\omega_f) - \exp(\omega_f - \Sigma_1)) - \exp(c_1(L)).
\end{split}
\end{equation}
Restricting \eqref{E614} to $\CH^1(X)$ yields
\begin{equation}\label{E615}
-f^* K_C = \Sigma_1 - c_1(L)
\end{equation}
and hence
\begin{equation}\label{E616}
L \isom f^* \Omega_C \otimes \CO_X\left(\Sigma_1\right).
\end{equation}
Therefore,
\begin{equation}\label{E617}
\begin{split}
&\quad \chi(X) - \chi(F)\chi(C) =
\frac{\omega_f^2}{2} - \frac{K_X^2 - 2c_2(X)}{2}\\
&= \ch(\CO_{\Sigma_0}) + (\omega_f \Sigma_1 - \Sigma_1^2)\\
&= \ch(\CO_{\Sigma_0}) + \sum_{p\in C} K_X \big(X_p - (X_p)_\text{red}\big) 
- \sum_{p\in C} \big(X_p - (X_p)_\text{red}\big)^2\\
&= \ch(\CO_{\Sigma_0}) + \sum_{p\in C} 2\Big(p_a(X_p) - p_a\big((X_p)_\text{red}\big)\Big) 
= \eta
\end{split}
\end{equation}
by restricting \eqref{E614} to $H^4(X)$.

\begin{defn}\label{DEFETA}
Let $G$ be an effective divisor on a smooth surface $X$.
Then the {\em $\eta$-invariant} of $G$ is
\begin{equation}\label{E620}
\begin{split}
\eta(G) 
&= \underbrace{\sum_{q\in G} \eta_0(G, q)}_{\eta_0(G)} + \underbrace{K_X (G - G_\text{red})
	+ (G^2 - G_\text{red}^2)}_{\eta_1(G)} 
\\
&= \underbrace{\sum_{q\in G} \eta_0(G,q)}_{\eta_0(G)} + 
\underbrace{2\big(p_a(G) - p_a(G_\text{red})\big)}_{\eta_1(G)}
\end{split}
\end{equation}
with $\eta_0(G, q)$ defined by 
\begin{equation}\label{E621}
\nu_q = \dim_\BC\left(
\CO_{X,q}/(g_\text{red} \frac{\partial (\log g)}{\partial x}, 
g_\text{red} \frac{\partial (\log g)}{\partial y})\right), 
\end{equation}
where $g(x,y)$ and $g_\text{red}(x,y)$ are the local defining
equations of $G$ and $G_\text{red}$ at $q$, respectively, i.e.,
\begin{equation}\label{E622}
g_\text{red}(x,y) = g_1(x,y) g_2(x,y) ... g_l(x,y)
\end{equation}
if $G$ is given by $g(x,y) = g_1^{m_1}(x,y) g_2^{m_2}(x,y) ... g_l^{m_l}(x,y) = 0$
in $X\isom \Delta_{xy}^2$ at $q$ with $g_1(x,y), g_2(x,y), ..., g_l(x,y)\in \BC[[x,y]]$ distinct and irreducible.
\end{defn}

We can summarize our previous discussion in the following proposition:

\begin{prop}\label{PROPETA}
Let $f: X\to C$ be a surjective morphism from a smooth projective surface $X$ to a
smooth projective curve $C$. Then
\begin{equation}\label{E641}
\chi(X) - \chi(F)\chi(C) = \sum_{p\in C} \eta(X_p)
\end{equation}
where $F$ is a general fiber of $f$.
\end{prop}

It is not hard to see that the $\eta$-invariant of a fiber is always
nonnegative: $\eta_0(G)\ge 0$ obviously and $\eta_1(G)\ge 0$
by the following lemma.

\begin{lem}\label{LEMETA}
Let $G\subset X$ be an effective divisor on a normal surface $X$. Suppose that
$X$ is $\BQ$-Gorenstein, $G$ and $G_\text{red}$ are Cartier,
$G^2 = 0$ and $G$ is nef. The following holds:
\begin{itemize}
\item For a smooth surface $Y$ and a proper birational map $f: Y\to X$,
\begin{equation}\label{E618}
\eta_1(f^* G) \ge \eta_1(G).
\end{equation}
\item $\eta_1(G) \ge 0$ if $X$ is smooth.
\end{itemize}
\end{lem}

\begin{proof}
First we show that \eqref{E618} holds when $K_Y E \ge 0$ for all exceptional curves
$E$ of $f$, i.e., $K_Y$ is $f$-nef. Let us consider
\begin{equation}\label{E667}
M = f^* G_\text{red} - (f^* G)_\text{red}.
\end{equation}
Clearly, $f^*G \ge M\ge 0$ and $f_* M = 0$. 
Therefore,
\begin{equation}\label{E645}
\begin{split}
\eta_1(f^*G) - \eta_1(G) &= - 2p_a((f^* G)_\text{red})
+ 2 p_a(G_\text{red})\\
&= (K_X G_\text{red} + G_\text{red}^2)
- (K_Y (f^* G)_\text{red} + (f^* G)_\text{red}^2)\\
&= K_Y M - M^2
\end{split}
\end{equation}
since $p_a(f^*G) = p_a(G)$ and $M . f^* G_\text{red} = 0$, obviously.

By our hypotheses on $G$, $F^2 \le 0$ for all effective divisors $F$ supported on
$f^* G$. Therefore, $M^2 \le 0$. And since $K_Y M \ge 0$, we conclude that
$\eta_1(f^*G) \ge \eta_1(G)$ if $K_Y$ is $f$-nef.

Next, we claim that \eqref{E618} holds when both $X$ and $Y$ are smooth.
It suffices to prove it for $f: Y\to X$ the blowup of $X$ at a point. This is more or
less obvious since $M = m E$ in \eqref{E645}, where $E$ is the exceptional divisor
of $f$. This proves \eqref{E618}.

If $K_X E \ge 0$ for all components $E\subset G$, then
\begin{equation}\label{E646}
\eta_1(G) = K_X(G - G_\text{red}) - G_\text{red}^2 \ge 0
\end{equation}
since $G_\text{red}^2 \le 0$.

If $X$ is smooth and $K_X E < 0$ for some component $E\subset G$, then $E$ is a $(-1)$-curve and
it can be blown down via $\pi: X\to X'$. Then $\eta_1(G)\ge \eta_1(G')$ by \eqref{E618},
where $G' = \pi_* G$. By induction, we conclude that $\eta_1(G) \ge 0$ if $X$ is smooth.
\end{proof}

Neither Proposition \ref{PROPETA} nor Lemma \ref{LEMETA} is new \cite{T3}.

Now back to \eqref{E605}, as a consequence of Proposition \ref{PROPETA},
we see that it holds provided that we can prove
\begin{equation}\label{E619}
\sum_{p\in C} \eta(X_p) \ge \delta - 28 b + 4 h^1(\CO_\E) + K_X\E
\end{equation}
which is a local statement on the singularities of $Y$.

By Theorem \ref{THMDELTAP}, we know that $Y$ has a double point of type \eqref{E642}
at the vertex of each cubic cone $S_p$. And by Lemma \ref{LEMETA}, all the local contributions
to $\eta(X_p)$ is nonnegative. Therefore, we can further reduce \eqref{E619} to
the following:
\begin{equation}\label{E648}
\eta(X_p \cap \tau^{-1}(U_q)) \ge \delta_p - 28 \lceil b_{p,q}\rceil
+ 4h^1(\CO_{\E_p}) + K_X \E_p
\end{equation}
where $U_q$ is an analytic open neighborhood of the vertex $q$ of $S_p$ in $W$, $b_{p,q}$ and $\E$ are defined in \eqref{E630}
and \eqref{E929}, respectively, and $\E_p$ is the connected component of $\E$ contained in $X_p$.

\subsection{Estimate $h^1(\CO_\E)$ and $K_X\E$}

Let us give a more explicit upper bound for the right hand side of \eqref{E648}. This involves estimation of the terms $h^1(\CO_\E)$ and $K_X\E$.

\begin{lem}\label{LEMTORSIONFREER1FOX}
Let $f: X\to C$ be a flat projective morphism from a projective surface $X$ to a smooth projective curve $C$. 
Suppose that
\begin{itemize}
\item
the general fibers of $f$ are smooth,
\item
$X$ is a local complete intersection, and
\item
$X$ is smooth at a point $q$ if the fiber $X_{f(q)}$ is nonreduced at $q$.
\end{itemize}
Then
$R^1 f_* \CO_X$ is torsion-free. Consequently,
all of $h^0(\CO_{f^* M})$, $h^1(\CO_{f^*M})$,
$h^0(\omega_X\otimes \CO_{f^* M})$ and $h^1(\omega_X\otimes \CO_{f^*M})$
are constants depending only on $\deg M$ for all effective divisors $M$ on $C$.
\end{lem}

\begin{proof}
Here we do not have to assume that $f$ has connected fibers.

This is clear if all fibers of $f$ are reduced. If $X_p$ is reduced, $h^0(\CO_{X_p})$ is the number of connected components of $X_p$ so $h^0(\CO_{X_t})$ is constant for $t$ in an open neighborhood of $p$; by flatness and invariance of Hilbert polynomials, $h^1(\CO_{X_t})$ is also locally constant at $p$. So $R^1 f_* \CO_X$ is torsion-free if all fibers of $f$ are reduced.

Let $\phi: \widehat{X}\to X$ be a birational morphism 
consisting of the blowups with smooth centers
over the nonsingular locus of $X$.
Then $\phi_* \CO_{\widehat{X}} = \CO_X$ and $R^j \phi_* \CO_{\widehat{X}} = 0$ for $j \ge 1$. Hence
\begin{equation}\label{E964}
R^1 (f\circ \phi)_* \CO_{\widehat{X}} = R^1 f_*(\phi_* \CO_{\widehat{X}}) = R^1 f_* \CO_X
\end{equation}
by Grothendieck spectral sequence. So we may simply replace $X$ by $\widehat{X}$. In particular, we can make $X$ into a surface such that all fibers of $X/C$ have simple normal crossing supports
outside of $X_\text{sing}$. By our hypothesis on $X$, $X_p$ has simple normal crossing support along its nonreduced components.

By stable reduction, there exists a finite morphism $B\to C$ from a smooth projective curve $B$ to $C$ such that the normalization $W$ of $X\times_C B$ has reduced fibers over $B$. So we have the diagram
\begin{equation}\label{E965}
\begin{tikzcd}
W \ar{r}{\varphi} \ar{d}[left]{e} & X\ar{d}{f}\\
B\ar{r}{\nu} & C
\end{tikzcd}
\end{equation}
Since $\varphi: W\to X$ is finite, $R^j \varphi_* \f = 0$ for all $j\ge 1$ and coherent sheaves $\f$ on $W$. Therefore,
\begin{equation}\label{E966}
R^1 (f\circ \varphi)_* \CO_W = R^1 f_* (\varphi_* \CO_W).
\end{equation}
By our hypotheses, $X$ is Cohen-Macaulay and smooth in codimension one. Therefore, $X$ is normal by Serre's criterion on normality and hence
$\CO_X$ is a direct summand of $\varphi_*\CO_W$. Then by \eqref{E966}, $R^1 f_* \CO_X$ is torsion free if $R^1 (f\circ \varphi)_* \CO_W$ is.

On the other hand, since $\nu$ is finite, we have
\begin{equation}\label{E967}
R^1 (f\circ \varphi)_* \CO_W = R^1 (\nu\circ e)_* \CO_W
= \nu_* (R^1 e_* \CO_W).
\end{equation}
Since $e: W\to B$ has reduced fibers, $R^1 e_* \CO_W$ is torsion-free. It follows that $\nu_* (R^1 e_* \CO_W)$ and hence
$R^1 (f\circ \varphi)_* \CO_W$ are torsion-free. This proves that $R^1 f_* \CO_X$ is torsion-free.

For an effective divisor $M$ on $C$,
applying $f_*$ to the exact sequence
\begin{equation}\label{E963}
\begin{tikzcd}
0 \ar{r} & \CO_X(-f^*M) \ar{r} & \CO_X \ar{r} & \CO_{f^*M} \ar{r} & 0,
\end{tikzcd}
\end{equation}
we obtain
\begin{equation}\label{E968}
\begin{tikzcd}[column sep=12pt]
0 \ar{r} & f_*\CO_X \otimes \CO_C(-M) \ar{r} & f_* \CO_X \ar{r} & f_* \CO_{f^*M}\\
\ar{r} & R^1 f_* \CO_X\otimes \CO_C(-M) \ar{r} & R^1 f_* \CO_X
\ar{r} & R^1 f_* \CO_{f^*M} \ar{r} & 0.
\end{tikzcd}
\end{equation}
Since $R^1 f_* \CO_X \otimes \CO_C(-M)$ is torsion-free and $f_* \CO_{f^* M}$ is torsion,
the above long exact sequence breaks down to two short exact sequences:
\begin{equation}\label{E975}
\begin{tikzcd}[column sep=small]
0 \ar{r} & f_* \CO_X \otimes \CO_C(-M) \ar{r} & 
f_* \CO_X \ar{r} & 
f_* \CO_{f^* M}
 \ar{r} & 0
\end{tikzcd}
\end{equation}
and
\begin{equation}\label{E976}
\begin{tikzcd}[column sep=small]
0 \ar{r} 
& R^1 f_* \CO_X\otimes \CO_C(-M) \ar{r} & R^1 f_* \CO_X \ar{r}
& R^1 f_* \CO_{f^* M } \ar{r} & 0.
\end{tikzcd}
\end{equation}
By \eqref{E975}, we have the identity
\begin{equation}\label{E977}
\begin{aligned}
h^0(\CO_{f^* M }) &= h^0(f_* \CO_{f^* M}) = c_1(f_* \CO_{f^*M})\\
&= c_1(f_*\CO_X) - c_1(f_*\CO_X \otimes \CO_C(-M)),
\end{aligned}
\end{equation}
whose right hand side only depends on $\deg M$. Therefore, $h^0(\CO_{f^* M })$ is a constant only depending on $\deg M$. 
Similarly, by \eqref{E976}, the same holds for $h^1(\CO_{f^*M})$.
Then by Serre duality, both
$h^0(\omega_X\otimes \CO_{f^* M})$ and $h^1(\omega_X\otimes \CO_{f^*M})$
are constants only depending on $\deg M$.
\end{proof}

For two effective divisors $F_1$ and $F_2$ on a smooth variety, we have the exact sequence
\begin{equation}\label{E969}
\begin{tikzcd}[column sep=small]
0 \ar{r} & \CO_{F_1}(-F_2)\ar[equal]{d} \ar{r} & \CO_{F_1+F_2}
\ar[equal]{d} \ar{r} & \CO_{F_2} \ar[equal]{d} \ar{r} & 0\\
& \displaystyle{\frac{\CO(-F_2)}{\CO(-F_1-F_2)}}
& \displaystyle{\frac{\CO}{\CO(-F_1-F_2)}}
& \displaystyle{\frac{\CO}{\CO(-F_2)}}
\end{tikzcd}
\end{equation}
So we have a surjection $H^1(\CO_{F_1 + F_2}) \twoheadrightarrow H^1(\CO_{F_2})$ on a smooth surface. Consequently, for two effective divisors $D_1$ and $D_2$ on a smooth surface with $D_1\le D_2$, we always have
$h^1(\CO_{D_1}) \le h^1(\CO_{D_2})$.

Applying this observation to 
\begin{equation}\label{E955}
\E \le M = \sum_{\delta_p \ne 0} \lceil b_{p,q} \rceil X_p,
\end{equation}
we obtain $h^1(\CO_\E) \le h^1(\CO_M)$. By Lemma \ref{LEMTORSIONFREER1FOX}, $h^1(\CO_M) = b h^1(F) = 5b$ and hence
$h^1(\CO_\E) \le 5b$. We can do better with the following lemma.

\begin{lem}\label{LEMH1FP}
Under the same hypotheses of Lemma \ref{LEMTORSIONFREER1FOX}, we further assume that $f$ has connected fibers. Suppose that
$X_p = \f + \Y$ for a fiber $X_p$ of $f$ over $p\in C$ , where
\begin{itemize}
\item
$\f$ and $\Y$ are two connected effective divisors contained in $X_p$,
\item
$\f$ and $\Y$ meet properly, and
\item
$X$ is smooth along $\f$.
\end{itemize}
Then
\begin{equation}\label{E956}
h^1(\CO_{b\f}) \le b\big( h^1(\CO_{X_p}) - h^1(\CO_\Y) -
\f_r\Y_r + 1
\big) 
\end{equation}
for all $b\in \BN$, where $\f_r$ and $\Y_r$ are the reduced effective divisors supported on $\f$ and $\Y$, respectively.
\end{lem}

\begin{proof}
Applying \eqref{E969} to $\f$ and $(m-1)\f$, we obtain
the exact sequence
\begin{equation}\label{E973}
\begin{tikzcd}[column sep=small]
0 \ar{r} & \CO_\f(-(m-1)\f) \ar[equal]{d} \ar{r} & \CO_{m\f} \ar{r} & 
\CO_{(m-1) \f} \ar{r} & 0\\
& 
\CO_{\f}((m-1) \Y)
\end{tikzcd}
\end{equation}
for $m\ge 1$. Clearly, $h^1(\CO_\f((m-1)\Y)) \le h^1(\CO_\f)$ by the exact sequence
\begin{equation}\label{E634}
\begin{tikzcd}[column sep=small]
0 \ar{r} & \CO_\f
\ar{r} & \CO_\f((m-1)\Y) \ar{r} &
\CO_{\f\cap (m-1)\Y}((m-1)\Y) \ar{r} & 0.
\end{tikzcd}
\end{equation}
Therefore,
\begin{equation}\label{E972}
h^1(\CO_{b\f}) 
\le \sum_{m=1}^{b} h^1(\CO_{\f}((m-1) \Y))
\le b h^1(\CO_\f).
\end{equation}

We can bound $h^1(\CO_\f)$ using the exact sequence
\begin{equation}\label{E974}
\begin{tikzcd}
0 \ar{r} & \CO_{X_p} \ar{r} & \CO_\f \oplus \CO_\Y \ar{r} 
& \CO_{\f\cap \Y} \ar{r} & 0
\end{tikzcd}
\end{equation}
with the induced long exact sequence
\begin{equation}\label{E978}
\begin{tikzcd}
& H^0(\CO_\f) \oplus H^0(\CO_\Y) \ar{r} &
H^0(\CO_{\f\cap \Y})\\
\ar{r} & H^1(\CO_{X_p}) \ar{r}
& H^1(\CO_\f) \oplus H^1(\CO_\Y) \ar{r} & 0.
\end{tikzcd}
\end{equation}
We obtain immediately from \eqref{E978} that $h^1(\CO_\f) \le 
h^1(\CO_{X_p}) - h^1(\CO_\Y)$. We can do a little better by observing that the map $H^0(\CO_{\f\cap \Y}) \to H^1(\CO_{X_p})$ is not necessarily zero. More precisely,
\begin{equation}\label{E979}
\begin{aligned}
h^1(\CO_\f) &= h^1(\CO_{X_p}) - h^1(\CO_\Y)
\\
&\quad - \dim \coker(H^0(\CO_\f) \oplus H^0(\CO_\Y) \to H^0(\CO_{\f\cap \Y})).
\end{aligned}
\end{equation}
By the diagram
\begin{equation}\label{E971}
\begin{tikzcd}[column sep=small]
& & H^0(\CO_{\f}) \oplus H^0(\CO_\Y)\ar{r} \ar{d} & H^0(\CO_{\f\cap \Y}) \ar[two heads]{d}\\
0\ar{r} & H^0(\CO_{X_{p,r}}) \ar[equal]{d} \ar{r} & H^0(\CO_{\f_r}) \oplus H^0(\CO_{\Y_r}) \ar[equal]{d} \ar{r} & H^0(\CO_{\f_r\cap \Y_r}) \ar[equal]{d}\\
& \BC & \BC\oplus \BC & \BC^{\oplus \f_r \Y_r}
\end{tikzcd}
\end{equation}
we conclude that
\begin{equation}\label{E981}
\begin{aligned}
&\quad \dim \coker(H^0(\CO_\f) \oplus H^0(\CO_\Y) \to H^0(\CO_{\f\cap \Y}))\\
&\ge \dim \coker(H^0(\CO_{\f_r}) \oplus H^0(\CO_{\Y_r}) \to H^0(\CO_{\f_r\cap \Y_r})) = \f_r \Y_r - 1
\end{aligned}
\end{equation}
where $X_{p,r}$ is the reduced effected divisor supported on $X_p$.
Then \eqref{E956} follows from \eqref{E972}, \eqref{E979}
and \eqref{E981}.
\end{proof}

We do not really need the term $\f_r \Y_r - 1$ in our application of \eqref{E956}.

From now on, we take $X$ to be the minimal desingularization of $Y$ only at the vertices of the cones $S_p$. Clearly, this does not change anything on either side of \eqref{E648} and $X$ satisfies the hypotheses of Lemma \ref{LEMTORSIONFREER1FOX} and \ref{LEMH1FP}.

Fixing a point $p\in C$ with $\delta_p\ne 0$, we write
\begin{equation}\label{E951}
X_p = \f_p + \Y_p
\end{equation}
as the sum of two effective divisors $\f_p$ and $\Y_p$
with $\Y_p$ the proper transform of $Y_p$.
Clearly, both $\f_p$ and $\Y_p$ are connected, $\f_p$ and $\Y_p$ meet properly and $X$ is smooth along $\f_p$.

Obviously, $\E_p \le \lceil b_{p,q} \rceil \F_p$ and hence
$h^1(\CO_{\E_p}) \le h^1(\CO_{\lceil b_{p,q} \rceil \F_p})$.
So we have
\begin{equation}\label{E970}
\begin{aligned}
h^1(\CO_{\E_p})
&\le h^1(\CO_{\lceil b_{p,q} \rceil \F_p})
\le \lceil b_{p,q} \rceil \big(h^1(\CO_{X_p}) - h^1(\CO_{\Y_p})\big)\\
&= \lceil b_{p,q} \rceil \big(5 - h^1(\CO_{\Y_p})\big)
\end{aligned}
\end{equation}
by Lemma \ref{LEMTORSIONFREER1FOX} and \ref{LEMH1FP}.

We claim that $h^1(\CO_{\Y_p})\ge 1$. Since $\Y_p$ is a partial normalization of $Y_p$, this is clear if $g(Y_p) \ge 1$. If $g(Y_p) = 0$, then by the second last statement of Theorem \ref{THMDELTAP}, $Y_p$ has a double point of type $y^2 = x^n$ at $q$ for some $n\le 6$ and hence $p_a(\Y_p) \ge p_a(Y_p) - 3 = 2$, i.e., $h^1(\CO_{\Y_p})\ge 2$. In conclusion, $h^1(\CO_{\Y_p})\ge 1$ and it follows that
\begin{equation}\label{E946}
h^1(\CO_{\E_p}) \le  4 \lceil b_{p,q} \rceil.
\end{equation}

Next, let us give a bound for $K_X \E_p$. Again, since $\E_p \le \lceil b_{p,q} \rceil \f_p$,
\begin{equation}\label{E961}
\begin{aligned}
K_X \E_p &\le \lceil b_{p,q} \rceil K_X \f_p 
= \lceil b_{p,q} \rceil (K_X X_p - K_X \Y_p)\\
&= \lceil b_{p,q} \rceil (8 - (2 p_a(\Y_p) - 2) - \f_p \Y_p)
\le 7 \lceil b_{p,q} \rceil
\end{aligned}
\end{equation}
as $p_a(\Y_p) = h^1(\CO_{\Y_p})\ge 1$ and $\f_p \Y_p \ge 1$. Indeed, as we will see later from the explicit desingularization of $Y$ at $q$, we always have $\f_p \Y_p \ge 2$. So the bound $7 \lceil b_{p,q} \rceil$ can be improved to $6 \lceil b_{p,q} \rceil$. But we have no need for it.

Finally, combining \eqref{E946} and \eqref{E961}, we have further reduced \eqref{E648} to
\begin{equation}\label{E962}
\eta(X_p \cap \tau^{-1}(U_q)) \ge \delta_p - 5 \lceil b_{p,q}\rceil.
\end{equation}
A proof of \eqref{E962} concludes our main theorem. So it comes down to the study of a double surface singularity
of type \eqref{E642}.

\section{Double Surface Singularities}\label{SECDSS}

\subsection{Resolution of Double Surface Singularities}

It suffices to prove the following:

\begin{prop}\label{PROPDOUBLESURFACE}
Let $Y = \{ y^2 = g(x,t) \}\subset \Delta_{xyt}^3$ be a normal surface singularity
with $g(x,t)\in \BC[[x,t]]$ satisfying
\begin{equation}\label{E925}
g(x,0) \not\equiv 0,\ g(0,t) = t^\delta \text{ and }
\left.\frac{\partial g}{\partial x}\right|_{x=0} \equiv 0 
\end{equation}
for some $\delta \in \BZ^+$
and let $\tau: X\to Y$ be the minimal resolution of $Y$. Then
\begin{equation}\label{E952}
\eta(X_0) \ge \delta -  5b
\end{equation}
where $X_0 = \tau^* Y_0 = \tau^*(Y\cap \{t=0\})$ and $b$ is defined by
\begin{equation}\label{E953}
b = \max_{\tau_* E = 0} \left(0, -\frac{a(E, Y)}{\mult_E(X_0)}\right).
\end{equation}
Here $X$ and $Y$ are regarded as families of curves over $\Delta_t$.
\end{prop}

Let us first resolve the double surface singularity
$Y = \{ y^2 = g(x,t) \}$.
The following procedure works for all $g(x,t)$
satisfying $g(0,t)\not\equiv 0$.

The simplest way to resolve the singularity of $Y$ is to consider it
as a double cover over $U = \Delta_{xt}^2$ ramified along
the curve $R = \{ g(x,t) = 0\}$, find a suitable embedded resolution
$(\widehat{U}, \widehat{R})$ of $(U,R)$ and then construct the double cover
of $\widehat{U}$ ramified along $\widehat{R}$. Actually, the algorithm
is very simple and straightforward:
Find the ``minimal'' birational map $u: \widehat{U} \to U$ such that
\begin{itemize}
\item $\widehat{U}$ and all components of $u^* R$ are smooth;
\item if the components $G_1, G_2, ..., G_m$ of $u^* R$ meet at a point, all but
at most one of $G_1, G_2, ..., G_m$ have even multiplicity in $u^* R$.
\end{itemize}
A more terse way to put this is that $u: \widehat{U} \to U$ is the minimal birational
map with $\widehat{U}$ smooth such that the divisor
\begin{equation}\label{E635}
\widehat{R} = u^* R - 2 \left\lfloor \frac{u^* R}{2} \right\rfloor
\end{equation}
is smooth, i.e., a disjoint union of smooth components. Then we take
$\widehat{Y}$ to be the double cover $\widehat{U}$ ramified along $\widehat{R}$.
We have the diagram
\begin{equation}\label{E636}
\begin{tikzcd}
\widehat{Y} \arrow{r}{\nu} \arrow{d}[left]{\widehat{\pi}} & Y\arrow{d}[right]{\pi} \\
\widehat{U} \arrow{r}{u} & U
\end{tikzcd}
\end{equation}

\begin{rem}\label{REMDOUBLE}
As an example, let us see how an $A_{\delta-1}$ singularity
$y^2 = x^2 + t^\delta$ is resolved in this way. 
The following diagram shows how to resolve $y^2 = x^2 + t^4$ as a double cover:
\begin{equation}\label{FIGCANONICAL}
\xy
{(30,20)*+{Y}},
{(5,20)*+{\widehat{Y}}},
{(5,5); (5,-5); **\crv{(0,0)}},
{(5,13); (5,3); **\crv{(0,8)}},
{(5,-13); (5,-3); **\crv{(0,-8)}},
{(0,16); (5,11); **@{-}},
{(0,-16); (5,-11); **@{-}},
{\ar (10,0); (20,0)^\nu},
{(25,4); (30,-1); **@{-}},
{(25,-4); (30,1); **@{-}},
{\ar (5,-18); (5, -23)_{\widehat{\pi}}},
{\ar (30,-18); (30, -23)_{\pi}},
{(5,-25); (5,-35); **\crv{(0,-30)}},
{(5,-33); (5,-43); **\crv{(0,-38)}},
{(0,-46); (5,-41); **@{-}},
{(0,-27); (8,-27); **@{-}},
{(0,-30); (8,-30); **@{-}},
{(-3,-28)*+{\widehat{R}\ \{}},
{\ar (10,-38); (20,-38)^u},
{(30,-50)*+{U}},
{(5,-50)*+{\widehat{U}}},
{(25, -36); (35,-36); **\crv{ (30,-41) }},
{(25, -41); (35,-41); **\crv{(30,-36)}},
{(30,-36);(30,-41); **@{-}},
{(39,-38.5)*+{\big\}\ R}},
\endxy
\end{equation}
\end{rem}

\subsection{Minimal Resolution of $Y$}

Such $\widehat{Y}$ is not necessarily the minimal resolution of $Y$ as it might contain $(-1)$-curves.
But these $(-1)$-curves are easily located and contraction of these curves will render $\widehat{Y}$ minimal.

\begin{prop}\label{PROPDOUBLEMIN}
Let $Y = \{ y^2 = g(x,t) \}\subset \Delta_{xyt}^3$ be a normal surface singularity with $g(x,t)\in \BC[[x,t]]$ satisfying
$g(x,0)\not\equiv 0$ and let $\nu: \widehat{Y}\to Y$ be the resolution of $Y$ constructed as above. Then
\begin{itemize}
\item
An exceptional curve $E\subset \widehat{Y}$ of $\nu$ is a smooth rational curve of self intersection $E^2 = -1$ if and only if $F = \widehat{\pi}(E)$ lies in the ramification locus $\widehat{R}$ and $F^2 = -2$. Hence the $(-1)$-curves of $\widehat{Y}$ are disjoint.
\item
$X$ is the minimal resolution of $Y$ for the contraction $\widehat{Y}\to X$ of all $(-1)$-curves of $\widehat{Y}$.
\end{itemize}
\end{prop}

\begin{proof}
For an exceptional curve $F\subset \widehat{U}$, if $F\subset \widehat{R}$, then
$E = \widehat{\pi}^{-1}(F)$ is a smooth rational curve of $E^2 = F^2/2$.
If $F\not\subset \widehat{R}$, we have three cases:
\begin{enumerate}
\item $E = \widehat{\pi}^{-1}(F)$ is an integral curve mapped $2$-to-$1$ to $F$ by $\widehat{\pi}$ and hence $E^2 = 2F^2$.
\item $\widehat{\pi}^* F = E_1 + E_2$, where $E_1$ and $E_2$ are two distinct integral curves satisfying $E_1 E_2 > 0$. Then
$E_i^2 = E_i . \widehat{\pi}^* F - E_1 E_2 = F^2 - E_1 E_2 < -1$ for $i=1,2$.
\item $\widehat{\pi}^* F = E_1 + E_2$, where $E_1$ and $E_2$ are two disjoint smooth rational curves satisfying $E_i^2 = F^2$ for $i=1,2$. This happens when $\widehat{\pi}$ is totally unramified over $F$.
\end{enumerate}
In the three cases above, only the third case will give us $(-1)$-curves.
Therefore, $(-1)$ curves $E$ on $\widehat{Y}$ are either pre-images of $(-2)$-curves in $\widehat{R}$, as stated in the proposition, or components of $\widehat{\pi}^{-1}(F)$ for $(-1)$-curves $F\subset \widehat{U}$, when $\widehat{\pi}$ is totally unramified over $F$.

If it is the latter, since $\widehat{\pi}$ is totally unramified over $F$, $F$ is disjoint from $\widehat{R}$. In other words, every component of $u^* R$ meeting $F$ has even multiplicity in $u^* R$. Since $F^2 = -1$, it can be contracted
via $\widehat{U}\to U'$, which is factored through by $u$.
That is, we have the diagram
\begin{equation}\label{E900}
\begin{tikzcd}
\widehat{U} \ar{r}{f} \ar[bend right]{rr}[below]{u} & U' \ar{r}{g} & U
\end{tikzcd}
\end{equation}
Since $F$ only meets components of $u^* R$ with even multiplicities, $F$ does not meet the proper transform of $R$. So under the contraction $f: \widehat{U}\to U$ of $F$, $f_* (u^* R) = g^* R$ remains to have smooth components; furthermore, it still has the property that no two components of odd multiplicities meet in $g^* R$. This contradicts the assumption that $u$ is the minimal birational map with these properties. In conclusion, $E$ is a $(-1)$-curve of $\widehat{Y}$ if and only if $F = \widehat{\pi}(E)$ is a $(-2)$-curve contained in $\widehat{R}$.

Next, let us prove that $X$ is minimal after we contract all $(-1)$-curves of $\widehat{Y}$ under $\widehat{Y}\to X$. Otherwise, there exist a smooth rational curve $A\subset \widehat{Y}$ and $(-1)$-curves $E_1, E_2, ..., E_m\subset \widehat{Y}$ such that $A^2 = -m-1\le -2$ and $A.E_i = 1$ for $i=1,2,...,m$.
Let $B = \widehat{\pi}(A)$ and $F_i = \widehat{\pi}(E_i)$ for $i=1,2,...,m$. By what we have proved above, each $F_i$ is a $(-2)$-curve contained in $\widehat{R}$
and hence $B$ meets at least one component of $\widehat{R}$ transversely.
So $\widehat{\pi}$ maps $A$ to $B$ $2$-to-$1$. And since both $A$ and $B$ are smooth rational curves, 
$B$ meets exactly two components of $\widehat{R}$ transversely and hence
$B.\widehat{R} = 2$. Therefore, $m\le 2$ as $F_i\subset \widehat{R}$. So we have $-m-1=A^2 = 2B^2 \ge -3$. It follows that
$A^2 = -2$, $B^2 = -1$ and $m=1$. Hence $B$ meets two components of $\widehat{R}$: one is $F = F_1$ and let us call the other $G$.

The configuration of $G\cup B \cup F$ is very simple. We have $G.B = B.F=1$ and
$G.F = 0$ since both $F$ and $G$ are components of $\widehat{R}$ and hence disjoint. Moreover, $B^2=-1$, $F^2=-2$ and $B$ meets no components of $\widehat{R}$ other than $F$ and $G$.

We can blow down the two curves $B$ and $F$ in sequence and again have a diagram like \eqref{E900}, where $f: \widehat{U}\to U'$ contracts $B\cup F$. Since $B$ meets no components of $\widehat{R}$ other than $F$ and $G$ and $B\cup F$ meets $G$ transversely at a unique point, $g^* R$ has again smooth components and the required property that no two components of odd multiplicities meet in $g^* R$. This once more contradicts the assumption that $u$ is the minimal birational map with these properties.
\end{proof}

\subsection{Discrepancies of $\nu$}

From \eqref{E636}, we can figure out the discrepancies of $\nu: \widehat{Y}\to Y$.
Since
\begin{equation}\label{E637}
\begin{split}
K_{\widehat{Y}}
&= \widehat{\pi}^* K_{\widehat{U}} + \frac{1}2 \widehat{\pi}^* \widehat{R}\\
\omega_Y &= \pi^* K_U + \frac{1}2 \pi^* R\\
K_{\widehat{U}} &= u^* K_U + \sum a(E, U) E,
\end{split}
\end{equation}
we obtain
\begin{equation}\label{E638}
\begin{split}
K_{\widehat{Y}}
&= \widehat{\pi}^* u^* K_U + \sum a(E, U) \widehat{\pi}^* E + \frac{1}2 \widehat{\pi}^* \widehat{R}\\
&= \nu^* \omega_Y + \sum a(E, U) \widehat{\pi}^* E
+ \frac{1}2 \widehat{\pi}^* (\widehat{R} - u^* R)\\
&= \nu^* \omega_Y + \sum a(E, U) \widehat{\pi}^* E 
- \widehat{\pi}^* \left\lfloor \frac{u^* R}{2} \right\rfloor.
\end{split}
\end{equation}
Clearly, the exceptional divisors of $\nu$ is contained in
$\sum \widehat{\pi}^*E $ for $E\subset \widehat{U}$ the exceptional divisors of $u$.

If $\mult_E(u^* R)$ is odd, $\widehat{\pi}$ is ramified along $E$ and
hence $\widehat{\pi}^* E = 2F$ for some $F$ with discrepancy
\begin{equation}\label{E639}
a(F, Y) = 2 a(E, U) - \mult_E(u^* R) + 1.
\end{equation}
If $\mult_E(u^* R)$ is even, $\widehat{\pi}$ is unramified at a general point
of $E$ and each component $F$ of $\widehat{\pi}^* E$ has discrepancy
\begin{equation}\label{E640}
a(F, Y) = a(E, U) - \frac{\mult_E(u^* R)}2.
\end{equation}

\begin{rem}\label{REMPROPDOUBLE}
For the canonical singularity $Y = \{ y^2 = x^2 + t^\delta\}$, it is obvious that
$\eta(\widehat{Y}_0) = \delta$
(see \eqref{FIGCANONICAL}). It is tempting to think that 
$\eta(\widehat{Y}_0)$ should go up as the singularity gets ``worse'' and thus expect that
$\eta(\widehat{Y}_0)\ge \delta$ under the hypotheses of the proposition. However,
this is simply false: consider $Y = \{ y^2 = x^4 + t^4\}$, which is resolved by
the following diagram:
\begin{equation}\label{FIGQUARTIC}
\xy
{(27.5,10)*+{Y}},
{(5,10)*+{\widehat{Y}}},
{(27.5,-31)*+{U}},
{(5,-31)*+{\widehat{U}}},
{(5,5); (5,-5); **\crv{(0,0)}},
{(0,8); (5,3); **@{-}},
{(0,-8); (5,-3); **@{-}},
{\ar (10,0); (20,0)^\nu},
{(0,0)*+{F}},
{(30,5); (30,-5); **\crv{(25,0)}},
{(25,5); (25,-5); **\crv{(30,0)}},
{\ar (5,-8); (5, -13)_{\widehat{\pi}}},
{\ar (27.5,-8); (27.5, -13)_{\pi}},
{(5,-15); (5,-25); **@{-}},
{(2,-17); (8,-17); **@{-}},
{(2,-18); (8,-18); **@{-}},
{(2,-19); (8,-19); **@{-}},
{(2,-20); (8,-20); **@{-}},
{(0,-28); (6,-22); **@{-}},
{(7,-25)*+{E}},
{\ar (10,-20); (20,-20)^u},
{(-1,-18)*+{\widehat{R}\ \{}},
{(27.5,-15); (27.5,-25); **@{-}},
{(25.5,-17); (29.5,-23); **@{-}},
{(25.5,-23); (29.5,-17); **@{-}},
{(25.5,-19); (29.5,-21); **@{-}},
{(25.5,-21); (29.5,-19); **@{-}},
{(33,-20)*+{\Big\}\ R}},
\endxy
\end{equation}
where $F$ is a component of $\widehat{Y}_0$ and a double cover of $E\isom \PP^1$
ramified at the $4$ points $\widehat{R}\cap E$. Clearly, $\eta(\widehat{Y}_0) = 2 < 4$.
On the other hand, $F$ is an elliptic curve satisfying $a(F, Y) = -1$ and hence
$b = 1$. Therefore, we have $\eta(\widehat{Y}_0) + 5b > \delta$. This example shows
that the term $5b$ in \eqref{E952} cannot be omitted.
\end{rem}

\subsection{A Special Sequence of Blowups}

Let us prove Proposition \ref{PROPDOUBLESURFACE}. We will work with $\widehat{Y}$ instead of the minimal resolution $X$ of $Y$. By Proposition \ref{PROPDOUBLEMIN}, $\widehat{Y}$ is not necessarily the same as $X$ and $X$ is obtained from $\widehat{Y}$ by contracting the $(-1)$-curves on $\widehat{Y}$. Among the three numbers $\eta(X_0)$, $\delta$ and $b$ in \eqref{E952}, this only affects $\eta(X_0)$, i.e., $\eta(X_0) \ne \eta(\widehat{Y}_0)$. We need to keep this in mind when estimating $\eta(X_0)$. On the other hand, $b$ is a birational invariant.

The birational map $u: \widehat{U} \to U$ consists of a sequence of blowups at points. The order of blowups can be altered in certain ways. Let us choose a special sequence of blowups to be the first $m$ blowups
\begin{equation}\label{E954}
\begin{tikzcd}[column sep=small]
\widehat{U} \ar{r}
\ar[bend left]{rrrrrr}[above]{u}
 & U^{[m]} \ar{r} & ... \ar{r} & 
U^{[i]} \ar{r} \ar[bend right]{rrr}[below]{u_i} & ... \ar{r} &
U^{[1]} \ar{r} & U^{[0]} \ar[equal]{r} & U
\end{tikzcd}
\end{equation}
where $U^{[i]}\to U^{[i-1]}$ is the blowup at the point
$p_{i-1} = \{x/t^{i-1} = t = 0\}$ with $E_i\subset U^{[i]}$ the exceptional divisor and $R_i\subset U^{[i]}$
the proper transform of $R$ for $i=1,2,...,m$.

That is,
$U^{[1]}$ is the blowup of $U$ at the origin $p_0 = \{x=t=0\}$ with 
exceptional divisor $E_1$. 
The map $u: \widehat{U}\to U$ factors through $U^{[1]}$ if and only if
$R$ is singular at $p_0$. 
Suppose that $R$ is singular at $p_0$ and $u$ factors through $U^{[1]}$. We further blow up
$U^{[1]}$ at the point $p_1 = \{x/t = t = 0\}\in E_1$ to obtain $U^{[2]}$ with exceptional divisor $E_2$. Then $u$ factors through $U^{[2]}$ if and only if $R_1$ is singular at $p_1$. In this way, we obtain a sequence of blowups in \eqref{E954} and we choose $m$ to be the largest integer
such that $u$ factors through $U^{[m]}$, which is equivalent to saying that $R_m$ either fails to pass through or is smooth at $p_m = \{x/t^{m} = t = 0\}$. 

We have the commutative diagram
\begin{equation}\label{E671}
\begin{tikzcd}
\widehat{Y} \arrow{rd}{\widehat{\pi}_m} \arrow{rr}{\nu} \arrow{d}[left]{\widehat{\pi}} & & Y\arrow{d}{\pi}\\
\widehat{U} \arrow{r} \arrow[bend right]{rr}[below]{u} & U^{[m]} \arrow{r}{u_m} & U.
\end{tikzcd}
\end{equation}
The central fiber $U_0^{[m]}$ of $U^{[m]}$ is
\begin{equation}\label{E669}
U_0^{[m]} = U_0 + E_1 + E_2 + ... + E_{m},
\end{equation}
where we abuse the notations a little bit by using $U_0$ and $E_i$ for the proper transforms of $U_0$ and $E_i$ in all of $\widehat{U}$ 
and $U^{[j]}$. For convenience, we also write $E_0 = U_0$.
Let
\begin{equation}\label{E674}
u_m^* R = R_m + c_1 E_1 + c_2 E_2 + ... + c_m E_m
\end{equation}
where $c_j = \mult_{E_j} (u_m^* R)$ satisfies
\begin{equation}\label{E683}
c_j - c_{j-1} = \mult_{p_{j-1}} R_{j-1}
\end{equation}
for $j=1,2,...,m$ (we let $c_0 = 0$).

We claim that the sequence $\{c_j\}$ is increasing and
\begin{equation}\label{E679}
c_1 \ge c_2 - c_1 \ge c_3 - c_2 \ge ... \ge c_m - c_{m-1} \ge 2.
\end{equation}
Indeed, we can interpret $c_j$ algebraically. Let
\begin{equation}\label{E665}
g(x,t) = \prod_{i=1}^n (x - \phi_i(t))
\end{equation}
for some $\phi_i(t) \in \BC[[\sqrt[N]{t}]$. Then
\begin{equation}\label{E957}
c_j = \sum_{i=1}^n \min(j, o(\phi_i)),
\end{equation}
where $o(\phi_i)$ is the order of $\phi_i(t)$ at $t=0$.
We may define $c_j$ for all $j\in \BZ$ by \eqref{E957} (hence $c_{-1} = -n$).

From \eqref{E957}, we see that $\{c_j - c_{j-1}\}$ is non-decreasing. 
Indeed, the difference $(c_{i} - c_{i-1}) - (c_{i+1} - c_{i})$ can be interpreted geometrically by
\begin{equation}\label{E982}
(c_{i} - c_{i-1}) - (c_{i+1} - c_{i}) = E_i R_m
\end{equation}
for $i=1,2,...,m-1$.
Since $R_{j-1}$ is singular at $p_{j-1}$, $c_j - c_{j-1}\ge 2$ for $j=1,2,...,m$ by \eqref{E683}. So we have \eqref{E679}. 

Since either $R_m$ is smooth at $p_m$ or $p_m\not\in R_m$,
i.e., $\mult_{p_m}(R_m) \le 1$, we have
\begin{equation}\label{E958}
c_m = \begin{cases}
\delta - 1 & \text{if } p_m\in R_m\\
\delta & \text{if } p_m\not\in R_m.
\end{cases}
\end{equation}
In addition,
the fact that $g_x(0,t)\equiv 0$ implies that $R_m$ and $E_m$ meet at $p_m$ with multiplicity at least $2$ when $R_m$ passes through $p_m$. That is,
\begin{equation}\label{E684}
\mult_{p_m} R_m = 1 \text{ and }
(R_m E_m)_{p_m}\ge 2 \text{ if } c_m = \delta - 1.
\end{equation}

Note that $E_m$ is a component satisfying
\begin{equation}\label{E668}
a(E_m, U) = m,\ \mult_{E_m} (\widehat{U}_0) = 1
\text{ and } \mult_{E_m} (u^* R) = c_m
\end{equation}
by \eqref{E669} and \eqref{E674}. Combining \eqref{E668} with 
\eqref{E639} and \eqref{E640}, we obtain
\begin{equation}\label{E670}
b \ge -\left(a(E_m,U) - \left\lfloor\frac{\mult_{E_m} (u^* R)}{2}\right\rfloor\right) = \left\lfloor \frac{c_m}2\right\rfloor - m.
\end{equation}

Let
$\rho: \widehat{Y}\to X$ be the contraction of all $(-1)$-curves of $\widehat{Y}$ given in Proposition \ref{PROPDOUBLEMIN} with diagram
\begin{equation}\label{E983}
\begin{tikzcd}
& \widehat{Y} \ar{d}{\nu} \ar{ld}[above]{\rho} \ar{r}{\widehat{\pi}} & \widehat{U}
\ar{d}{u}\\
X \ar{r}{\tau} & Y \ar{r}{\pi} & U
\end{tikzcd}
\end{equation}
Let $F_i = \widehat{\pi}^{-1}(E_i)$ for $E_i\subset \widehat{U}$ and $i=0,1,...,m$. That is, $F_i$ are the proper transforms of $E_i$.  Next, we are going to figure out which $F_i$ are contracted by $\rho$. Let
\begin{equation}\label{E960}
\I = \{ 1\le i \le m-1: \rho_* F_i = 0\}.
\end{equation}
By Proposition \ref{PROPDOUBLEMIN}, $F_i$ is contracted by $\rho$ if and only if $c_i = \mult_{E_i}(\widehat{U}_0)$ is odd and $E_i^2 = -2$ in $\widehat{U}$. When we consider such $E_i$ in $U^{[m]}$, it has the property that $2\nmid c_i$, $2\mid c_{i-1}, c_{i+1}$ and
$E_i\cap R_m = \emptyset$; otherwise, if one of $c_{i-1}$ and $c_{i+1}$ is odd, we need to blow up the intersection $E_{i-1}\cap E_i$ or $E_i\cap E_{i+1}$, which renders $E_i^2 < -2$ in $\widehat{U}$; similarly, if $E_i$ meets $R_m$, we need to blow up the intersections $E_i\cap R_m$, which again renders $E_i^2 < -2$ in $\widehat{U}$. Since
\begin{equation}\label{E686}
E_i R_m = \begin{cases}
(c_{i} - c_{i-1}) - (c_{i+1} - c_i) & \text{ if } i < m\\
c_{m} - c_{m-1} & \text{ if } i = m
\end{cases}
\end{equation}
we conclude
\begin{equation}\label{E643}
\begin{aligned}
\I &= \{ 1\le i \le m-1: \rho_* F_i = 0\}
\\
&= \Big\{ 1\le i \le m-1: 2\nmid c_i \text{ and } 2\nmid (c_i - c_{i-1}) = (c_{i+1} - c_i)\Big\}.
\end{aligned}
\end{equation}

The criterion for $\rho_* F_m = 0$ is more complicated. Suppose that $\rho_* F_m = 0$. Of course, $c_m$ is odd.
If $E_m$ meets $R_m$ at more than one point, it is easy to see that
$\rho_* F_m \ne 0$. So $E_m$ meets $R_m$ at a unique point $q$. Let $U'$ be the blowup of $U^{[m]}$ at $q$ with exceptional divisor $E'$. Then $\widehat{U}\to U^{[m]}$ factors through $U'$. If $(E_m R_m)_q > \mult_q (R_m)$, the proper transforms of $E_m$ and $R_m$ meet in $U'$ and hence $\rho_* F_m \ne 0$. Consequently,
\begin{equation}\label{E986}
\mult_q (R_m) = (E_m R_m)_q = E_mR_m = c_m - c_{m-1}
\end{equation}
by \eqref{E686} and hence
\begin{equation}\label{E985}
\mult_{E'} (u')^* R = \begin{cases}
c_{m-1} + c_m + \mult_q(R_m) = 2c_m & \text{if } q=q_m\\
c_m + \mult_q(R_m) = 2c_m - c_{m-1} & \text{if } q\ne q_m
\end{cases}
\end{equation}
for $q_m = E_{m-1}\cap E_m$ and $u': U'\to U$. Clearly, $\mult_{E'} (u')^* R$ must be even; otherwise, $\rho_* F_m \ne 0$. This, together with $2\nmid c_m$ and \eqref{E986}, is also sufficient for $\rho_* F_m = 0$. Namely,
\begin{equation}\label{E987}
\begin{aligned}
\rho_* F_m = 0 &\Leftrightarrow 2\nmid c_m,\ \mult_q (R_m) = c_m - c_{m-1} \text{ for some } q\in E_m\\
&\hspace{24pt} \text{and either } q=q_m
\text{ or } q\ne q_m \text{ and } 2\mid c_{m-1}.
\end{aligned}
\end{equation}

If it is the former case $q = q_m$,
$E_{m-1} R_m \ge \mult_q (R_m)$. Therefore,
\begin{equation}\label{E989}
c_{m-1} - c_{m-2} \ge 2(c_m - c_{m-1})
\end{equation}
if $\rho_* F_m = 0$ and $q_m\in E_m\cap R_m$. In addition, in this case, we see that
\begin{equation}\label{E991}
\rho_* \widehat{\pi}_m^* E_m \ne 0 \text{ is nonreduced in } X_0
\end{equation}
due to the presence of $E'$, which is nonreduced in $U_0'$.

If it is the latter case $q\ne q_m$, we have a better estimate for $b$ than \eqref{E670} using $E'$. That is,
\begin{equation}\label{E988}
\begin{aligned}
b &\ge -\left(a(E', U) - \frac{\mult_{E'}(u')^* R}{2}\right)
= \frac{2c_m - c_{m-1}}2 - m - 1\\
&\ge
c_m - \frac{c_m - 3}2 - m - 1 = \left\lfloor\frac{c_m}2\right\rfloor +1 - m
\end{aligned}
\end{equation}
if $\rho_* F_m = 0$ and $q_m \not\in E_m\cap R_m$. 

\subsection{Estimate $\eta(X_0)$}

When $\delta = 1$, $Y$ is smooth at $p_0$ and hence $X = Y$. Since $Y_0$ is singular at $p_0$, $\eta(X_0) \ge 1$ and \eqref{E952} follows. Let us assume that $\delta \ge 2$. Then $m\ge 1$.

We use the notation $\eta_0(X_0, G)$ to denote
\begin{equation}\label{E680}
\eta_0(X_0, G) = \sum_{q\in G} \eta_0(X_0, q)
\end{equation}
for a subset $G\subset X_0$. Let $q_i = E_{i-1}\cap E_i$ in $U^{[m]}$ for $1\le i \le m$ and let us consider $\eta_0(X_0,Q_i)$
for
\begin{equation}\label{E984}
Q_i = \rho(\widehat{\pi}_m^{-1}(q_i)).
\end{equation}
We discuss this in three not mutually exclusive cases.
\begin{enumerate}
\item
Suppose that $c_{i-1}$ and $c_i$ are both even.
Then $\rho_* F_j \ne 0$ for $j=i-1,i$. It is not hard to see that
$Q_i$ either contains a union $G$ of curves such that $G\cap \rho(F_{i-1}) \ne G\cap \rho(F_i) \ne \emptyset$ or consists of two points in $\rho(F_{i-1})\cap \rho(F_i)$. Either way, $Q_i$ contains at least two singularities of $X_{0,r} = (X_0)_\text{red}$. Hence
\begin{equation}\label{E647}
\eta_0(X_0, Q_i) \ge 2.
\end{equation}
\item
Suppose that $c_{i-1}$ and $c_i$ are both odd and $\rho_* F_i \ne 0$. 
Then $\rho_* F_j \ne 0$ for $j=i-1,i$.
Since $c_{i-1}$ and $c_i$ are odd,
$\rho(F_{i-1})$ and $\rho(F_i)$ are two disjoint components of $X_0$;
then $Q_i$ contains a union $G$ of curves such that $G\cap \rho(F_{i-1}) \ne G\cap \rho(F_i) \ne \emptyset$ and \eqref{E647} follows. Note that $\rho_* F_{i}\ne 0$ automatically holds
by \eqref{E643}
if $2\nmid c_{i-1}c_i$ and $i < m$.
\item
For every $1\le i\le m$, we claim that $Q_i$ contains at least one singularity of $X_{0,r}$ and hence 
\begin{equation}\label{E901}
\eta_0(X_0,Q_i) \ge 1.
\end{equation}
This is clear if $\rho_* F_j \ne 0$ for $j=i-1,i$. Otherwise, $\rho$ only contracts one of $F_{i-1}$ and $F_i$.
Again, \eqref{E901} is obvious when $\rho_* F_{i-1} = 0$ or $\rho_* F_i = 0$
and $i < m$. This leaves us the only case that $i=m$ and $\rho_* F_m = 0$. By \eqref{E686}, $E_m R_m\ge 2$
in $U^{[m]}$. Since we need to further blow up the intersections
$E_m\cap R_m$ under $\widehat{U}\to U^{[m]}$, 
it is easy to see that $Q_m$ contains at least one singularity of
$X_{0,r}$. Therefore, \eqref{E901} holds.
\end{enumerate}

In addition, when $c_m = \delta - 1$, $\rho_* F_m \ne 0$ and
there is at least one singularity of $X_{0,r}$ in $\rho(\widehat{\pi}_m^{-1}(p_m))$
by \eqref{E684}. Therefore,
\begin{equation}\label{E673}
\rho_* F_m \ne 0 \text{ and }
\eta_0(X_0,\rho(\widehat{\pi}_m^{-1}(p_m))) \ge 1
\text{ if } c_m = \delta - 1.
\end{equation}

The two sets $Q_i$ and $Q_j$ for $1\le i < j \le m$
are disjoint unless $j=i+1$ and $\rho_* F_i = 0$, i.e., $i\in \I$.
So each $i\in \I$ reduces the above estimate of $\eta_0(X_0)$ by one. Therefore,
\begin{equation}\label{E959}
\eta_0(X_0) \ge \sum_{i=1}^{m} \eta_0(X_0,Q_i) - |\I| + (\delta - c_m)
\end{equation}
where the term $\delta - c_m$ accounts for the contribution of $p_m$ given in \eqref{E673} when $c_m = \delta - 1$.
Applying our previous estimate on $\eta_0(X_0, Q_i)$, we have
\begin{equation}\label{E939}
\begin{aligned}
\eta_0(X_0) &\ge \sum_{i=1}^{m} \eta_0(X_0,Q_i)
- |\I| + (\delta - c_m)
\\
&\ge m + \frac{1}2 \sum_{i=1}^{m} (1 + (-1)^{c_i - c_{i-1}}) 
- \varepsilon - |\I| + (\delta - c_m)
\end{aligned}
\end{equation}
where the first term $m$ accounts for each $i$ in \eqref{E901}, the second term takes into account the cases that $c_{i-1}$ and $c_i$ are both even or odd in \eqref{E647}, the third term
\begin{equation}\label{E990}
\varepsilon = \begin{cases}
1 & \text{if } 2\nmid c_{m-1}c_m \text{ and } \rho_* F_m = 0\\
0 & \text{otherwise}
\end{cases}
\end{equation}
corrects the previous term in the case that both $c_{m-1}$ and $c_m$ are odd and $\rho_* F_m = 0$, and the rest two terms $-|\I|$
and $\delta - c_m$ are explained as above.

Let us also keep in mind that $\eta(X_0)$ has contribution from $\eta_1(X_0)$. At the very least, if $X_0$ is nonreduced,
$X_0 - X_{0,r}\ne 0$ and hence
\begin{equation}\label{E681}
\eta_1(X_0) = K_X(X_0 - X_{0, r}) - (X_0 - X_{0, r})^2 \ge 2.
\end{equation}
Suppose that $c_m$ is odd. If $\rho_* F_m \ne 0$, then
$\rho_* \widehat{\pi}_m^* E_m \ne 0$ is nonreduced in $X_0$
and \eqref{E681} holds. Otherwise, $\rho_* F_m = 0$ and hence
$E_m$ meets $R_m$ at a unique point $q$ by \eqref{E987}.
If $q = q_m$, then $X_0$ is nonreduced by \eqref{E991} and
\eqref{E681} follows. Otherwise, $q\ne q_m$ and we have 
\eqref{E988}. In conclusion, if $c_m$ is odd, we have either 
\eqref{E681} or a better bound for $b$ in \eqref{E988}.
Either way, we can say at least
\begin{equation}\label{E685}
\eta_1(X_0) + 5b
\ge 5\left\lfloor\frac{c_m}2\right\rfloor - 5m 
+ (1 - (-1)^{c_m})
\end{equation}
by \eqref{E670}, \eqref{E988} and \eqref{E681}.
Combining \eqref{E939} and \eqref{E685}, we obtain
\begin{equation}\label{E672}
\begin{aligned}
\eta(X_0) + 5b &\ge 5\left\lfloor\frac{c_m}2\right\rfloor - 4m - |\I| + \frac{1}2 \sum_{i=1}^{m} (1 + (-1)^{c_i - c_{i-1}})\\
&\quad - \varepsilon + (\delta - c_m) + (1 - (-1)^{c_m}).
\end{aligned}
\end{equation}
Let $1\le l\le m$ be the largest integer such that $c_l - c_{l-1}$ is odd. Namely, we have
\begin{equation}\label{E675}
2\nmid \mu = c_l - c_{l-1} 
\text{ and }
2\mid (c_{i} - c_{i-1}) \text{ for } l+1\le i\le m.
\end{equation}
We let $l=0$ if $c_i - c_{i-1}$ are even for all $1\le i\le m$.
Then
\begin{equation}\label{E676}
c_m = \sum_{i=1}^m (c_i - c_{i-1}) = \sum_{i=1}^l (c_i - c_{i-1})
+ \sum_{i=l+1}^m (c_i - c_{i-1})\ge \mu l + 2(m-l)
\end{equation}
and
\begin{equation}\label{E677}
\frac{1}2 \sum_{i=1}^{m} (1 + (-1)^{c_i - c_{i-1}})
\ge \frac{1}2 \sum_{i=l+1}^{m} (1 + (-1)^{c_i - c_{i-1}})
= m-l.
\end{equation}
By the description of $\I$ in \eqref{E643}, we see that
$\I\subset \{1,2,...,l-1\}$ and $|i-j|\ge 2$ for all $i\ne j\in \I$. Consequently,
\begin{equation}\label{E682}
|\I| \le \left\lfloor \frac{l}2 \right\rfloor.
\end{equation}

Combining \eqref{E672}, \eqref{E677} and \eqref{E682}, we derive
\begin{equation}\label{E678}
\begin{aligned}
\eta(X_0) + 5b  
&\ge 5\left\lfloor\frac{c_m}2\right\rfloor - 4m
+ \frac{1}2 \sum_{i=1}^{m} (1 + (-1)^{c_i - c_{i-1}}) - |\I|\\
&\quad - \varepsilon + (\delta - c_m) + (1 - (-1)^{c_m})\\
&\ge 5\left\lfloor\frac{c_m}2\right\rfloor - 4(m-l) - 4l
+ (m-l) - \left\lfloor \frac{l}2 \right\rfloor\\
&\quad - \varepsilon + (\delta - c_m) + (1 - (-1)^{c_m})\\
&= 5\left\lfloor\frac{c_m}2\right\rfloor - \left\lfloor \frac{9l}2\right\rfloor - 3(m-l)\\
&\quad - \varepsilon + (\delta - c_m) + (1 - (-1)^{c_m}).
\end{aligned}
\end{equation}

We claim that the above inequality implies \eqref{E952}. We argue in three cases.

\smallskip

\noindent\underline{$l=0$}: In this case, all $c_i - c_{i-1}$ are even for $i=1,2,...,m$. Hence $c_m$ is even and $\varepsilon = 1 - (-1)^{c_m} = 0$.
Then \eqref{E678} becomes
\begin{equation}\label{E992}
\begin{aligned}
\eta(X_0) + 5b &\ge \frac{5c_m}2 - 3m + (\delta - c_m)\\
&\ge \frac{5c_m}2 - \frac{3c_m}2 + (\delta - c_m) = \delta.
\end{aligned}
\end{equation}
since $2m \le c_m$ by \eqref{E676}.

\smallskip

\noindent\underline{$l>0$ and $\varepsilon=0$}: Since $\mu\ge 3$, we have
\begin{equation}\label{E993}
\begin{aligned}
\eta(X_0) + 5b  
&= 5\left\lfloor\frac{c_m}2\right\rfloor - \left\lfloor \frac{9l}2\right\rfloor - 3(m-l)\\
&\quad + (\delta - c_m) + (1 - (-1)^{c_m})\\
&\ge 5\left\lfloor\frac{c_m}2\right\rfloor - \left\lfloor \frac{9(c_m - 2(m-l))}{2\mu}\right\rfloor - 3(m-l)\\
&\quad + (\delta - c_m) + (1 - (-1)^{c_m})\\
&\ge 5\left\lfloor\frac{c_m}2\right\rfloor - \left\lfloor \frac{3(c_m - 2(m-l))}{2}\right\rfloor - 3(m-l)\\
&\quad + (\delta - c_m) + (1 - (-1)^{c_m})\\
&= 5\left\lfloor\frac{c_m}2\right\rfloor - \left\lfloor \frac{3c_m}{2}\right\rfloor + (1 - (-1)^{c_m}) + (\delta - c_m)\\
&= c_m + (\delta - c_m) = \delta
\end{aligned}
\end{equation}
by \eqref{E676} and \eqref{E678}.

\smallskip

\noindent\underline{$l>0$ and $\varepsilon=1$}: In this case,
both $c_{m-1}$ and $c_m$ are odd and $\rho_* F_m = 0$.
By \eqref{E987}, $E_m$ meets $R_m$ at the unique point $q_m$. Then by \eqref{E989},
\begin{equation}\label{E994}
\mu \ge 2(c_m - c_{m-1}) \ge 4.
\end{equation}
And since $\mu$ is odd, we actually have $\mu \ge 5$. Then
\begin{equation}\label{E995}
c_m \ge \mu l + 2(m-l) \ge 5l + 2(m-l) \ge 3l + 2(m-l) + 2
\end{equation}
by \eqref{E676}. Then \eqref{E678} becomes
\begin{equation}\label{E996}
\begin{aligned}
\eta(X_0) + 5b  
&= \frac{5(c_m-1)}2 - \left\lfloor \frac{9l}2\right\rfloor - 3(m-l) + (\delta - c_m) + 1\\
&\ge \frac{5(c_m-1)}2 - \left\lfloor \frac{3(c_m - 2(m-l) - 2)}{2}\right\rfloor - 3(m-l)\\
&\quad + (\delta - c_m) + 1\\
&\ge \frac{5(c_m-1)}2 - \left\lfloor \frac{3c_m}{2}\right\rfloor + (\delta - c_m) + 4\\
&= \frac{5(c_m-1)}2 -  \frac{3c_m-1}{2} + (\delta - c_m) + 4\\
&= c_m + (\delta - c_m) + 2 = \delta + 2.
\end{aligned}
\end{equation}

\smallskip

This finishes the proof of \eqref{E952}.

\end{document}